\documentclass{amsart}
\usepackage{amsfonts,amssymb,stmaryrd,amscd,amsmath,latexsym,amsbsy}
\numberwithin{equation}{section}
\theoremstyle{plain}

\newtheorem{thm}{Theorem}[section]

\newtheorem{defthm}[thm]{Definition/Theorem}
\newtheorem{prop}[thm]{Proposition}

\newtheorem{defi}[thm]{Definition}
\newtheorem{lem}[thm]{Lemma}
\newtheorem{cor}[thm]{Corollary}

\theoremstyle{remark}
\newtheorem{rema}[thm]{Remark}

\title{Connection coefficients for basic Harish-Chandra series}
\author{Jasper V. Stokman}
\address{KdV Institute for Mathematics, University of Amsterdam,
Science Park 904, 1098 XH Amsterdam, The Netherlands \& IMAPP,
Radboud University Nijmegen, Heyendaalseweg 135, 6525 AJ Nijmegen, The Netherlands.}
\email{j.v.stokman@uva.nl}
\begin{document}
\begin{abstract}
Basic Harish-Chandra series are asymptotically free meromorphic
solutions of the system of basic hypergeometric
difference equations associated to root systems.
The associated
connection coefficients are explicitly computed in terms of 
Jacobi theta functions. We interpret the connection coefficients 
as the transition functions for asymptotically free meromorphic solutions of 
Cherednik's root system analogs of the quantum 
Knizhnik-Zamolodchikov equations. They thus
give rise to explicit elliptic solutions of root system 
analogs of dynamical Yang-Baxter and reflection equations.
Applications to quantum $c$-functions, basic hypergeometric
functions, reflectionless difference operators and multivariable Baker-Akhiezer
functions are discussed.
\end{abstract}
\maketitle

\section{Introduction}\label{Intro}
 
The monodromy of the Heckman-Opdam 
system of hypergeometric differential equations associated
to root systems is explicitly computed in \cite{HO}. 
The key step is the derivation of explicit expressions
of the connection coefficients for the Harish-Chandra series solution
of the system
in terms of Gamma functions. We prove the basic hypergeometric
analog of this result by determining explicit expressions of
the connection coefficients for basic 
Harish-Chandra series in terms of Jacobi theta functions.

The basic Harish-Chandra series
is a self-dual, meromorphic solution of the
system of basic hypergeometric difference equations associated to 
root systems, characterized by its plane wave asymptotics
deep in a fixed Weyl chamber. 
The system of basic hypergeometric difference equations is the spectral problem
of the commuting Ruijsenaars-Macdonald-Koorn\-win\-der-Cherednik
difference operators, whose Laurent polynomial
solutions are the celebrated symmetric Macdonald-Koornwinder polynomials. 

The difference Cherednik-Matsuo correspondence
relates the spectral problem 
to Cherednik's 
quantum affine Knizhnik-Zamolodchikov (KZ) equations
associated to minimal principal series representations 
of the affine Hecke algebra, which are root system analogs of
the Frenkel-Reshetikhin-Smirnov quantum KZ equations.
Thus the connection coefficients are transition functions
for asymptotically free meromorphic solutions of quantum KZ equations.
This point of view leads to the interpretation of
the connection matrices as elliptic solutions to root system analogs of
dynamical Yang-Baxter equations.

In the remainder of the introduction we give a detailed description of the
main results, including precise references to the literature.
We start by fixing some basic notations 
in Subsection \ref{idsection}. 
We discuss the two relevant compatible systems of difference equations in
Subsection \ref{IDE}. The basic Harish-Chandra
series are discussed in Subsection \ref{HCintro}. 
In Subsection \ref{connc} we formulate the associated
connection coefficients problem and give the explicit expressions of the
connection coefficients. In Subsection \ref{connc} we also discuss
the relation to modified dynamical Yang-Baxter equations.
Applications to
basic hypergeometric functions and $c$-functions, reflectionless basic
Harish-Chandra series and multivariable Baker-Akhiezer functions are
discussed in Subsections \ref{csection}, \ref{ReflSect} and
\ref{BAsection} respectively.

\subsection{Initial data}\label{idsection}
We start with the introduction of the initial data of the
Che\-red\-nik-Macdonald theory \cite{CB,M} on Macdonald-Koornwinder polynomials.
The setup follows closely the conventions of the recent exposition
\cite{StB}, which provides
a uniform framework for all the known cases of the theory.

The initial datum 
is given by a triple $(D,\kappa,q)$ 
with $D$ the root system datum, $\kappa$ the 
associated free parameters, and $0<q=e^\tau<1$ the deformation parameter. 
The root system datum $D=(R_0,\Delta_0,\bullet,\Lambda,\widetilde{\Lambda})$ 
consists of
\begin{enumerate}
\item a finite, reduced crystallographic 
root system $R_0$ in the Euclidean space 
$\bigl(E,\bigl(\cdot,\cdot\bigr)\bigr)$, irreducible within the Euclidean
subspace $V$ spanned by $R_0$, 
\item a basis $\Delta_0=(\alpha_1,\ldots,\alpha_n)$ of the root
system $R_0$,
\item $\bullet\in\{u,t\}$ (``u'' standing for 
untwisted and ``t'' for twisted),
\item full lattices $\Lambda,\widetilde{\Lambda}\subseteq E$ such that
\begin{equation*}
\begin{split}
&Q\subseteq\Lambda,\qquad \bigl(\Lambda,Q^\vee\bigr)
\subseteq\mathbb{Z},\\
&\widetilde{Q}\subseteq\widetilde{\Lambda},\qquad \bigl(\widetilde{\Lambda},
\widetilde{Q}^{\vee}
\bigr)\subseteq\mathbb{Z},
\end{split}
\end{equation*}
with $Q$ and $\widetilde{Q}$
(respectively $Q^\vee$ and $\widetilde{Q}^\vee$)
the root lattice (respectively co-root lattice)
of $R_0$ and of the dual root system 
\begin{equation*}
\widetilde{R}_0:=
\begin{cases}
R_0^\vee=\{\alpha^\vee:=\frac{2\alpha}{|\alpha|^2}\}_{\alpha\in R_0}
\quad &\hbox{ if }\,\,\bullet=u,\\
R_0 \quad &\hbox{ if }\,\, \bullet=t
\end{cases}
\end{equation*}
respectively.
\end{enumerate}
We write
\begin{equation*}
\mu_\alpha:=
\begin{cases}
1 \qquad &\hbox{ if }\,\, \bullet=u,\\
\frac{|\alpha|^2}{2} \qquad &\hbox{ if }\,\, \bullet=t
\end{cases}
\end{equation*}
and $\widetilde{\alpha}:=\mu_\alpha\alpha^\vee$ for $\alpha\in R_0$.
Then $\widetilde{R}_0=\{\widetilde{\alpha}\}_{\alpha\in R_0}$.
We write $R_0^+$ and $R_0^-$ for the positive and negative roots in $R_0$
with respect to the basis $\Delta_0$.

We attach to the root system datum $D$ an irreducible
affine root system $R(D)$. It is built from
the reduced affine root system
$R^\bullet:=\{\alpha^{(r)}:=
\mu_\alpha rc+\alpha\}_{r\in\mathbb{Z},\alpha\in R_0}$, where
$\mu_\alpha rc+\alpha$ stands for the affine linear function
$z\mapsto \mu_\alpha r+(\alpha,z)$ on $E$, by adding 
the multiple $2\alpha^{(r)}$ if
$\bigl(\Lambda,\alpha^\vee\bigr)=2\mathbb{Z}$.
A basis of $R(D)$ is obtained by adding to $\Delta_0$ 
the simple affine root $\alpha_0:=\mu_\psi c-\psi$
with $\psi$ the highest (respectively highest short)
root of $R_0$ relative to the basis $\Delta_0$ if 
$\bullet=u$ (respectively $\bullet=t$).

The free parameters $\kappa$ are represented by the function values
$\kappa_a$ ($a\in R(D)$) of a $W$-invariant
function $\kappa: R(D)\rightarrow \mathbb{R}$, where $W:=
W_0\ltimes\widetilde{\Lambda}$ is the extended affine Weyl group.
We call $\kappa$ a multiplicity function and set $\kappa_{2\alpha^{(r)}}:=
\kappa_{\alpha^{(r)}}$ if $2\alpha^{(r)}\not\in R(D)$. With this convention 
the values $\kappa_\alpha,\kappa_{2\alpha},\kappa_{\alpha^{(1)}},
\kappa_{2\alpha^{(1)}}$ ($\alpha\in R_0$) uniquely determine the
multiplicity function $\kappa$. 

Here are three important examples of root system data. The
$\textup{GL}_{n+1}$ root system datum is 
$D=(R_0,\Delta_0,\bullet,\mathbb{Z}^{n+1},
\mathbb{Z}^{n+1})$, where $R_0$ is the root system of type
$A_n$ with its standard realization in $E=\mathbb{R}^{n+1}$.
Cherednik \cite{CB} developed his theory on Macdonald polynomials
mainly for reduced semisimple
root system data, in which case $D=(R_0,\Delta,\bullet,
P,\widetilde{P})$ with $V=E$ and with the lattices taken to be
the weight lattices $P$ and $\widetilde{P}$ of $R_0$
and $\widetilde{R}_0$ respectively. Here semisimple refers to the fact that
$V=E$, reduced to the fact that $R(D)=R^\bullet$ and similarly
for the associated dual affine root system, see Subsection \ref{dars}.
The Koornwinder case of the Macdonald-Koornwinder theory
\cite{Koo,N,Sa,NSt,NSt} corresponds to the root system datum
$(R_0,\Delta_0,t,Q,Q)$ with $R_0$ of type $A_1$ or of type $B_n$
($n\geq 2$).

To clarify the link with well known families of 
one variable basic hypergeometric functions
it is often convenient to express formulas in terms of
Askey-Wilson (AW) type parameters and their duals.
They are defined as follows. For a fixed root $\alpha\in R_0$ 
the associated AW parameters are 
\begin{equation}\label{AW}
\{a_\alpha,b_\alpha,c_\alpha,d_\alpha\}:=
\{q^{\kappa_\alpha+\kappa_{2\alpha}},-q^{\kappa_\alpha-\kappa_{2\alpha}},
q_\alpha q^{\kappa_{\alpha^{(1)}}+\kappa_{2\alpha^{(1)}}},
-q_\alpha q^{\kappa_{\alpha^{(1)}}-\kappa_{2\alpha^{(1)}}}\}
\end{equation}
where $q_\alpha:=q^{\mu_\alpha}$.
The dual AW parameters are 
\begin{equation}\label{dualAW}
\{\widetilde{a}_\alpha,\widetilde{b}_\alpha,\widetilde{c}_\alpha,
\widetilde{d}_\alpha\}=
\{q^{\kappa_\alpha+\kappa_{\alpha^{(1)}}},-q^{\kappa_\alpha-
\kappa_{\alpha^{(1)}}}, q_\alpha q^{\kappa_{2\alpha}+
\kappa_{2\alpha^{(1)}}}, -q_\alpha q^{\kappa_{2\alpha}-
\kappa_{2\alpha^{(1)}}}\}.
\end{equation}
They only depend on the orbit $W_0\alpha$ of
$\alpha\in R_0$. 
The four (dual) AW parameters associated to a root 
$\alpha\in R_0$ comprise either one, two or four of the 
free parameters, reflecting
the fact that the associated local rank one reduction of the 
Macdonald-Koornwinder theory relates to the theory of 
continuous $q$-ultraspherical 
polynomials, continuous $q$-Jacobi polynomials and Askey-Wilson polynomials
respectively. The case at hand can be read off from the root system datum as 
follows.\\

\noindent
{\bf Continuous $q$-ultraspherical case:} 
$\bigl(\Lambda,\alpha^\vee\bigr)=\mathbb{Z}=\bigl(\widetilde{\Lambda},
\widetilde{\alpha}^\vee\bigr)$, then
\[\kappa_{2\alpha^{(1)}}=\kappa_{2\alpha}=\kappa_{\alpha^{(1)}}=\kappa_\alpha.
\]
{\bf Continuous $q$-Jacobi case:}
either $\bigl(\Lambda,\alpha^\vee\bigr)=\mathbb{Z}$ and 
$\bigl(\widetilde{\Lambda},\widetilde{\alpha}^\vee\bigr)=2\mathbb{Z}$, then
\[
\kappa_{2\alpha}=\kappa_\alpha\quad \hbox{ and }\quad
\kappa_{2\alpha^{(1)}}=\kappa_{\alpha^{(1)}},
\]
or $\bigl(\Lambda,\alpha^\vee\bigr)=2\mathbb{Z}$ and 
$\bigl(\widetilde{\Lambda},\widetilde{\alpha}^\vee\bigr)=\mathbb{Z}$, then
\[
\kappa_{\alpha^{(1)}}=\kappa_{\alpha}\quad \hbox{ and }\quad
\kappa_{2\alpha^{(1)}}=\kappa_{2\alpha}.
\]
{\bf Askey-Wilson case:}
$\bigl(\Lambda,\alpha^\vee\bigr)=2\mathbb{Z}=\bigl(\widetilde{\Lambda},
\widetilde{\alpha}^\vee\bigr)$.\\

\noindent
The Askey-Wilson case only occurs when $D$ is the Koornwinder root system
datum $D=(R_0,\Delta_0,t,Q,Q)$ with $R_0$ of type $A_1$ or of type
$B_n$ ($n\geq 2$) and $\alpha\in R_0$ a short root. For reduced semisimple
root system datum, one is dealing with the continuous $q$-ultraspherical case
for all roots $\alpha\in R_0$. Continuous $q$-Jacobi cases only occur
in the untwisted theory $\bullet=u$, see \cite{StB}.

\subsection{Integrable difference equations}\label{IDE}
Consider the trigonometric function
\begin{equation}\label{A(z)}
\begin{split}
A(z):=&\frac{(1-a_\psi q^{\psi(z)})(1-b_\psi q^{\psi(z)})(1-c_\psi q^{\psi(z)})
(1-d_\psi q^{\psi(z)})}{(1-q^{2\psi(z)})(1-q_\psi^2q^{2\psi(z)})}\\
\times&\prod_{\alpha\in R_0^+: (\widetilde{\psi},\widetilde{\alpha}^\vee)
=1}\frac{(1-a_\alpha q^{\alpha(z)})(1-b_\alpha q^{\alpha(z)})}
{(1-q^{2\alpha(z)})}
\end{split}
\end{equation}
in $z\in E_{\mathbb{C}}:=\mathbb{C}\otimes_{\mathbb{R}}E$, where we canonically
extend the (affine) roots to complex
affine linear functions on $E_{\mathbb{C}}$. 
The symmetric Macdonald-Koornwinder polynomials 
associated to the initial
datum $(D,\kappa,q)$ 
are trigonometric Laurent polynomial eigenfunctions of the
difference operator 
\begin{equation*}
\bigl(Lf\bigr)(z):=q^{-(\rho,\widetilde{\psi})}\sum_{w\in W_0/W_{0,\psi}}
A(w^{-1}z)\bigl(f(z+w\widetilde{\psi})-f(z)\bigr)+
\Bigl(\sum_{w\in W_0/W_{0,\psi}}q^{-(\rho,w\widetilde{\psi})}\Bigr)f(z)
\end{equation*}
acting on meromorphic functions $f(z)$ in $z\in E_{\mathbb{C}}$, where
\[
\rho:=\frac{1}{2}\sum_{\alpha\in R_0^+}(\kappa_\alpha+
\kappa_{\alpha^{(1)}})\widetilde{\alpha}^\vee,
\]
$W_0\subseteq\textup{GL}_{\mathbb{C}}(E_{\mathbb{C}})$
is the Weyl group of $R_0$ and
$W_{0,\psi}$ the stabilizer subgroup of $\psi$.

In fact, for the $\textup{GL}_{n+1}$ root system datum,
the difference operator $L$ is a quantum conserved integral of
Ruijsenaars' \cite{R} quantum relativistic integrable many body system.
For reduced semisimple root system datum, $L$ 
is the Macdonald \cite{Mpol} difference operator associated to a
quasi-miniscule weight. For the Koornwinder root system datum, $L$ is the
Koornwinder's \cite{Koo}
multivariable analog of the Askey-Wilson \cite{AW} second order difference
operator.
Higher order difference operators, mutually commuting and commuting with $L$,
have been constructed using the theory of double affine Hecke 
algebras, see, e.g., \cite{CB,M,StB}. We recall their construction in 
Subsection \ref{RMKC}. 
We call $L$ and the associated
higher order difference operators
Ruijsenaars-Macdonald-Koorn\-win\-der-Cherednik (RMKC) 
operators. 

The associated spectral problem is a compatible
system of basic hypergeometric difference equations. It is the
natural generalization of
the Heckman-Opdam \cite{HO} system of hypergeometric differential equations
associated to root systems to the basic hypergeometric level.
It has a natural upgrade to a bispectral problem, see Subsection 
\ref{biRMKC}.

The bispectral quantum KZ equations 
\begin{equation}\label{biqaKZ}
C_{(\tau(\nu),\tau(\lambda))}(z,\xi)f(z-\nu,\xi-\lambda)=f(z,\xi),\qquad
\nu\in\widetilde{\Lambda},\, \lambda\in\Lambda
\end{equation}
associated to the initial datum
$(D,\kappa,q)$ form an explicit compatible system of linear
difference equations
for meromorphic functions $f(z,\xi)$ in $(z,\xi)\in E_{\mathbb{C}}\times
E_{\mathbb{C}}$ taking values in a complex $\#W_0$-dimensional vector
space $\mathcal{V}$. Here $\tau(\nu)$ (respectively $\tau(\lambda)$)
stands for the element $\nu\in\widetilde{\Lambda}$ (respectively
$\lambda\in\Lambda$) viewed as element of 
$W=W_0\ltimes\widetilde{\Lambda}$ (respectively $\widetilde{W}:=
W_0\ltimes\Lambda$). The explicit expressions for
$C_{(\tau(\nu),\tau(\lambda))}(z,\xi)$ are given in Theorem \ref{BqKZ}. 
For the twisted theory $\bullet=t$ with $\widetilde{\Lambda}=\Lambda$, 
the bispectral quantum KZ equations
\eqref{biqaKZ} have been defined and studied before in
\cite{vM,vMS,StSph}. 

For fixed $\xi\in E_{\mathbb{C}}$,
the restricted compatible system of difference equations
\begin{equation}\label{qaKZ}
C_{(\tau(\nu),\tau(0))}(z,\xi)f(z-\nu)=f(z),\qquad \nu\in\widetilde{\Lambda}
\end{equation}
for $\mathcal{V}$-valued meromorphic functions $f(z)$ in $z\in E_{\mathbb{C}}$
are Cherednik's \cite{CQKZ,CQKZ2} quantum affine KZ equations associated
to the minimal principal series representation of the affine Hecke algebra with
central character $q^\xi$, see \cite{vMS,vM} for details
(here $q^\xi$ is interpreted as element of the complex
algebraic torus $\textup{Hom}(\widetilde{\Lambda},\mathbb{C}^*)$ by
$\nu\mapsto q^{(\nu,\xi)}$). 
For the $\textup{GL}_{n+1}$ root system datum,
the quantum affine KZ equations \eqref{qaKZ} become
Frenkel-Reshetikhin-Smirnov \cite{FR,Sm} type quantum KZ equations, which were
derived in \cite{FR} as the consistency
conditions satisfied by matrix coefficients of products of quantum affine
algebra intertwiners. {}From the physics point of view,
they form the consistency conditions of correlation functions 
for integrable two dimensional lattice models from statistal physics.
See \cite{EFK,JM} for detailed expositions.

The Cherednik-Matsuo correspondence \cite{Mat,CInt} relates
solutions of the affine KZ equations to solutions of the Heckman-Opdam system
of hypergeometric differential equations. Its difference 
analog \cite{CQKZ2,CInd,Ka,StI} embeds 
the solution space of the quantum affine 
KZ equations \eqref{qaKZ} into the solution space of a
spectral problem of the RMKC operators. See Subsection
\ref{CMK} for the definition of the associated embedding $\chi$.
The difference Cherednik-Matsuo correspondence
was obtained in \cite[Thm. 3.4(a)]{CQKZ2} for reduced
semisimple root datum (untwisted case), see also \cite{CInd}.
Subsequently Kato \cite[Thm. 4.6]{Ka} showed 
that $\chi$ maps solutions of the quantum affine KZ equations \eqref{qaKZ} 
to eigenfunctions of $L$ by different methods. The surjectivity of $\chi$ 
was claimed in \cite[Thm. 3.4(b)]{CQKZ2} and \cite[Thm. 4.3(b)]{CInd}.
It was proved for generic spectral parameter $\xi\in E_{\mathbb{C}}$
in \cite[Thm. 5.16(b)]{StI} 
using an extension of the methods from \cite{CInt,O}
for the classical Cherednik-Matsuo correspondence.

In Subsection \ref{CMK} we show that the difference Cherednik-Matsuo
correspondence gives rise to an embedding of the 
solution space of the bispectral quantum KZ equations to the 
solution space of a bispectral problem of the RMKC operators.
This extends the results from \cite{vM,vMS,StSph}, which dealt
with the twisted case.

\subsection{Basic Harish-Chandra series}\label{HCintro}

Extending the results from \cite{vMS,vM,StSph}, we prove in
Subsection \ref{SOLsection} the existence of an asymptotically free, meromorphic
solution $\Phi_{KZ}$ of the bispectral quantum KZ equations \eqref{biqaKZ} and
establish its basic properties (selfduality, description of
singularities). Via the difference Cherednik-Matsuo correspondence it 
leads to the existence of asymptotically
free meromorphic eigenfunctions of the RMKC operator $L$.
More precisely, we will establish the following result.
\begin{thm}\label{THM1}
There exists a unique meromorphic function 
$\Phi(\cdot,\cdot)=\Phi(\cdot,\cdot;D,\kappa;q)$ on $E_{\mathbb{C}}\times 
E_{\mathbb{C}}$ satisfying
\begin{enumerate}
\item the eigenvalue equations
\begin{equation}\label{EE}
L\Phi(\cdot,\xi)=
\Bigl(\sum_{w\in W_0/W_{0,\psi}}q^{\psi(w^{-1}\xi)}\Bigr)\Phi(\cdot,\xi),
\end{equation}
viewed as identity of meromorphic functions in $(\cdot,\xi)\in
E_{\mathbb{C}}\times E_{\mathbb{C}}$ (the unspecified first entry is to
emphasize that this is the space on which the RMKC operator $L$ is 
acting),
\item the asymptotic expansion 
\[
\Phi(z,\xi)=\frac{\mathcal{W}(z,\xi)}
{\mathcal{S}(z)\widetilde{\mathcal{S}}(\xi)}
\sum_{\alpha\in Q^+}
\Gamma_\alpha(\xi)q^{-\alpha(z)},\qquad Q^+:=\bigoplus_{i=1}^n\mathbb{Z}_{\geq 0}
\alpha_i
\]
with 
\begin{enumerate}
\item the plane wave $\mathcal{W}(z,\xi)=q^{(\rho-\xi,\widetilde{\rho}+w_0z)}$,
where $w_0\in W_0$ is the longest Weyl group element and 
\[\widetilde{\rho}:=
\frac{1}{2}\sum_{\alpha\in R_0^+}\bigl(\kappa_\alpha+\kappa_{2\alpha}\bigr)
\alpha^\vee
\]
is a dual version of $\rho$,
\item the series $\Psi(z,\xi):=
\sum_{\alpha\in Q^+}\Gamma_\alpha(\xi)q^{-\alpha(z)}$ converging
normally for $(z,\xi)$ in compacta of $E_{\mathbb{C}}\times E_{\mathbb{C}}$ (hence
it defines 
a holomorphic function in $(z,\xi)\in E_{\mathbb{C}}\times E_{\mathbb{C}}$),
\item the holomorphic functions $\mathcal{S}(\cdot)$ and 
$\widetilde{\mathcal{S}}(\cdot)$ on $E_{\mathbb{C}}$,
capturing the singularities of $\Phi(\cdot,\cdot)$, explicitly given by
\begin{equation*}
\begin{split}
\mathcal{S}(z)&:=\prod_{\alpha\in R_0^+}\bigl(q_\alpha^2a_\alpha^{-1}
q^{-\alpha(z)},q_\alpha^2b_\alpha^{-1}q^{-\alpha(z)},
q_\alpha^2c_\alpha^{-1}q^{-\alpha(z)},q_\alpha^2d_\alpha^{-1}q^{-\alpha(z)};
q_\alpha^2\bigr)_{\infty},\\
\widetilde{\mathcal{S}}(\xi)&:=\prod_{\alpha\in R_0^+}\bigl(q_\alpha^2
\widetilde{a}_\alpha^{-1}
q^{-\widetilde{\alpha}(\xi)},q_\alpha^2\widetilde{b}_\alpha^{-1}
q^{-\widetilde{\alpha}(\xi)},
q_\alpha^2\widetilde{c}_\alpha^{-1}q^{-\widetilde{\alpha}(\xi)},
q_\alpha^2\widetilde{d}_\alpha^{-1}q^{-\widetilde{\alpha}(\xi)};
q_\alpha^2\bigr)_{\infty},
\end{split}
\end{equation*}
where 
\[
\bigl(x_1,\ldots,x_m;q\bigr)_{\infty}:=\prod_{j=1}^m\prod_{i=0}^{\infty}
(1-q^ix_j),
\]
\item the normalization 
\[
\Gamma_0(\xi)=\prod_{\alpha\in R_0^+}\bigl(q_\alpha^2q^{-2\widetilde{\alpha}(\xi)};
q_\alpha^2\bigr)_{\infty}.
\]
\end{enumerate}
\end{enumerate}
\end{thm}
The function $\Phi(\cdot,\cdot)$ is the natural generalization to the
present basic hypergeometric context 
of Harish-Chandra series, see \cite{HO} and  \cite[Part 1]{HS}.
We therefore call $\Phi(\cdot,\cdot)$ the 
{\it basic Harish-Chandra series} associated to the initial datum 
$(D,\kappa,q)$. It is automatically an eigenfunction of the higher
order RMKC operators, see Subsection \ref{HCsub}.
Specializing $\xi$ to a polynomial spectral point
turns $\Phi(z,\xi)$ into the selfdual symmetric Macdonald-Koornwinder
polynomial associated to $(D,\kappa,q)$ (the proof from \cite[\S 4]{HS} and
\cite[\S 3.3]{StSph} 
generalizes easily to the present setup, with the Macdonald-Koornwinder
polynomials associated to $(D,\kappa,q)$ as defined in \cite{StB}).
\begin{rema}
If $R_0$ is of rank one then explicit expressions of the basic Harish-Chandra
series in terms of basic hypergeometric series are known, see
\cite[\S 5]{StSph}. For $R_0$ of rank two explicit basic hypergeometric
expressions are known only for the $\textup{GL}_{3}$ root system
datum, see \cite{NS}. Explicit 
expressions for the coefficients
$\Gamma_\alpha(\xi)$ ($\alpha\in Q^+$) of the power series expansion of
$\Phi(z,\xi)$ are only known in higher rank cases if $D$ is 
the $\textup{GL}_{n+1}$ system datum, see \cite{NS}.
\end{rema}

The above characterization of the basic Harish-Chandra series is 
easy to establish. Its existence was proved for $\bullet=t$
and $\Lambda=\widetilde{\Lambda}$ in
\cite{vMS,vM,StSph}. These methods are extended to the present context
in Section \ref{bHCs}.

The particular choice of normalization of the basic Harish-Chandra series
(see Theorem \ref{THM1}(2d)) 
is to ensure the selfduality of the basic Harish-Chandra series:
\begin{thm}\label{THM2}
Let $\widetilde{\Phi}(\cdot,\cdot)$
be the basic Harish-Chandra series associated to the dual initial datum
$(\widetilde{D},\widetilde{\kappa},q)$, where $\widetilde{D}:=
(\widetilde{R}_0,\widetilde{\Delta}_0,\bullet,\widetilde{\Lambda},\Lambda)$
with $\widetilde{\Delta}_0:=(\widetilde{\alpha}_1,\ldots,\widetilde{\alpha}_n)$
and with $\widetilde{\kappa}$ the dual set of free parameters
as defined in Subsection \ref{mf} (its associated AW parameters are the
dual AW parameters). Then
\[
\Phi(z,\xi)=\widetilde{\Phi}(\xi,z).
\]
\end{thm}
The selfduality of $\Phi$ implies that $\Phi$ solves a
bispectral problem, 
see Subsection \ref{HCsub}. We prove Theorem \ref{THM1} and Theorem \ref{THM2} in Subsection \ref{HCsub}. 

\begin{rema}
The Harish-Chandra series is not selfdual, but it
does solve a bi\-spectral problem. The associated bispectral problem is the bispectral extension of
the Heckman-Opdam system of hypergeometric differential equations by eigenvalue equations for the rational degenerations
of dual RMKC operators  (see \cite[Thm. 6.9]{ChBi} and the subsequent remark, and \cite[Thm. 6.12]{ChBi}).
\end{rema}

\subsection{Connection coefficients and root system analogs of
 elliptic $R$-matrices}
\label{connc}
The RMKC operators are $W_0$-equivariant, resulting
in the $W_0\times W_0$-invariance of the solution space 
of the associated bispectral problem. It leads to the following
definition of connection matrices.

\begin{defthm}\label{THM3}
Let $\mathcal{F}$ be the space of meromorphic 
$\widetilde{\Lambda}\times\Lambda$-translation invariant
meromorphic functions $f(z,\xi)$ in
$(z,\xi)\in E_{\mathbb{C}}\times E_{\mathbb{C}}$.
For $\sigma\in W_0$ there exists a unique matrix
\[
M^\sigma(z,\xi)=\bigl(m_{\tau_1,\tau_2}^\sigma(z,\xi)\bigr)_{\tau_1,\tau_2\in W_0}
\]
with coefficients $m_{\tau_1,\tau_2}^\sigma$ in $\mathcal{F}$ such that
\[
\Phi(\sigma^{-1}z,\tau_2^{-1}\xi)=\sum_{\tau_1\in W_0}m_{\tau_1,\tau_2}^\sigma(z,\xi)
\Phi(z,\tau_1^{-1}\xi)
\]
as meromorphic functions in $(z,\xi)\in E_{\mathbb{C}}\times E_{\mathbb{C}}$.
We call $M^\sigma$ the connection matrix associated to $\sigma\in W_0$
and $M:=\{M^\sigma\}_{\sigma\in W_0}$ the connection cocycle.
\end{defthm}
For fixed $\xi\in E_{\mathbb{C}}$
the connection cocycle is Cherednik's monodromy cocycle 
\cite[Cor. 5.3]{CQKZ} (see also \cite[\S 4]{CInd}) for the quantum
affine KZ equations \eqref{qaKZ}, 
represented as matrix with respect to
a suitable basis of asymptotically free solutions.
We will first establish the theorem in the context of the bispectral
quantum KZ equations. Applying the difference Cherednik-Matsuo 
correspondence then 
provides the current formulation in terms of basic Harish-Chandra series.
See Subsections \ref{CMK} and \ref{HCsub} for the details. 

The cocycle property of the connection cocycle is
\begin{equation}\label{CP}
M^{\sigma_1\sigma_2}(z,\xi)=M^{\sigma_1}(z,\xi)M^{\sigma_2}(\sigma_1^{-1}z,\xi),
\qquad \sigma_1,\sigma_2\in W_0
\end{equation}
with $M^e$ the identity matrix, where $e\in W_0$ is the neutral element.
Note furthermore that 
$m_{\tau_1,\tau_2}^\sigma(z,\xi)=m_{\tau_2^{-1}\tau_1,e}^\sigma(z,\tau_2^{-1}\xi)$.

The 
following theorem provides 
explicit expressions of the entries of the connection
matrices in terms of theta functions.
Write
\[
\theta(x_1,\ldots,x_m;q):=\prod_{j=1}^m\bigl(x_j,q/x_j;q\bigr)_{\infty}
\]
for products of the normalized Jacobi theta function $\theta(x;q)$
in base $q$ and define for $\alpha\in R_0$ the following two meromorphic
functions in $(x,y)\in\mathbb{C}\times\mathbb{C}$,
\[
\mathfrak{e}_\alpha(x,y):=q^{-\frac{1}{2\mu_\alpha}
(\kappa_\alpha+\kappa_{2\alpha}-x)(\kappa_\alpha+\kappa_{\alpha^{(1)}}-y)}
\frac{\theta\bigl(\widetilde{a}_\alpha q^y, \widetilde{b}_\alpha q^y,
\widetilde{c}_\alpha q^y,d_\alpha q^{y-x}/\widetilde{a}_\alpha;
q_\alpha^2\bigr)}
{\theta\bigl(q^{2y},d_\alpha q^{-x};q_\alpha^2\bigr)}
\]
and its dual version
\[
\widetilde{\mathfrak{e}}_\alpha(x,y):=q^{-\frac{1}{2\mu_\alpha}
(\kappa_\alpha+\kappa_{\alpha^{(1)}}-x)(\kappa_\alpha+\kappa_{2\alpha}-y)}
\frac{\theta\bigl(a_\alpha q^y, b_\alpha q^y,
c_\alpha q^y,\widetilde{d}_\alpha q^{y-x}/a_\alpha;
q_\alpha^2\bigr)}
{\theta\bigl(q^{2y},\widetilde{d}_\alpha q^{-x};q_\alpha^2\bigr)}.
\]
The meromorphic functions $\mathfrak{e}_\alpha(x,y)$
and $\widetilde{\mathfrak{e}}_\alpha(x,y)$ only depend on $W_0\alpha$.
Furthermore, $\mathfrak{e}_\alpha(x,y)$ and
$\widetilde{\mathfrak{e}}_\alpha(x,y)$ 
are $2\mu_\alpha$-translation invariant in both $x$ and $y$,
which follows from repeated application of
the functional equation 
\begin{equation}\label{fe}
\theta(q^rx;q)=
(-q^{-\frac{1}{2}}x)^{-r}q^{-\frac{r^2}{2}}\theta(x;q),\qquad
r\in\mathbb{Z}.
\end{equation}
For $i\in\{1,\ldots,n\}$ let $s_i\in W_0$ be 
the orthogonal reflection associated to the simple root $\alpha_i$. 
\begin{thm}\label{THM4}
Fix $i\in\{1,\ldots,n\}$. Let $i^*\in\{1,\ldots,n\}$ such
that $\alpha_{i^*}=-w_0\alpha_i$. Then
$m_{\tau_1,\tau_2}^{s_i}\equiv 0$ if $\tau_1\not\in\{\tau_2,\tau_2s_{i^*}\}$ and
\begin{equation}\label{explicitm}
\begin{split}
m_{e,e}^{s_i}(z,\xi)&=\frac{\mathfrak{e}_{\alpha_i}(\alpha_i(z),
\widetilde{\alpha}_{i^*}(\xi))-\widetilde{\mathfrak{e}}_{\alpha_i}
(\widetilde{\alpha}_{i^*}(\xi),\alpha_i(z))}
{\widetilde{\mathfrak{e}}_{\alpha_i}(\widetilde{\alpha}_{i^*}(\xi),
-\alpha_i(z))},\\
m_{s_{i^*},e}^{s_i}(z,\xi)&=\frac{\mathfrak{e}_{\alpha_i}(\alpha_i(z),
-\widetilde{\alpha}_{i^*}(\xi))}{\widetilde{\mathfrak{e}}_{\alpha_i}(
\widetilde{\alpha}_{i^*}(\xi),-\alpha_i(z))}.
\end{split}
\end{equation}
\end{thm}
Note that it is not immediatedly clear that the right hand sides of
\eqref{explicitm} are $\widetilde{\Lambda}
\times\Lambda$-translation invariant since the $\mathfrak{e}_{\alpha}(x,y)$
are only $2\mu_\alpha$-translation invariant in $x$ and $y$. 
The $\widetilde{\Lambda}\times\Lambda$-translation invariance
can be verified directly using quadratic transformation formulas, 
see \cite[\S 7]{StQ}.

The proof of Theorem \ref{THM4}, which follows from rank reduction \cite{LS} 
and the explicit formulas \cite{StQ} for the connection coefficients
in rank one, is discussed in Section \ref{CC}. 

The expression for the nontrivial 
entries $m_{e,e}^{s_i}$ and $m_{s_{i^*},e}^{s_i}$ of the connection
matrix $M^{s_i}$ simplifies for the continuous
$q$-ultraspherical case $\bigl(\Lambda,\alpha_i^\vee\bigr)=
\mathbb{Z}=\bigl(\widetilde{\Lambda},
\widetilde{\alpha}_i^\vee\bigr)$ by a direct application of 
the so called addition formula \cite{WW} for theta functions
\begin{equation}\label{Rid}
\theta\bigl(x\lambda,x/\lambda,\mu\nu,\mu/\nu;q\bigr)-
\theta\bigl(x\nu,x/\nu,\lambda\mu,\mu/\lambda;q\bigr)=
\frac{\mu}{\lambda}\theta\bigl(x\mu,x/\mu,\lambda\nu,\lambda/\nu;q\bigr),
\end{equation}
which plays a fundamental role in the theory \cite[Chpt. 11]{GR}
of elliptic hypergeometric functions (see \cite[App. A]{Sp} for a detailed 
discussion of \eqref{Rid}). 
\begin{prop}\label{cUScase}
If $\bigl(\Lambda,\alpha_i^\vee\bigr)=\mathbb{Z}=\bigl(\widetilde{\Lambda},
\widetilde{\alpha}_i^\vee\bigr)$ then
\begin{equation*}
\begin{split}
m_{e,e}^{s_i}(z,\xi)&=q^{\frac{1}{\mu_i}(2\kappa_i-\widetilde{\alpha}_{i^*}(\xi))
\alpha_i(z)}\frac{\theta\bigl(a_i,q^{\widetilde{\alpha}_{i^*}(\xi)-\alpha_i(z)};q_i
\bigr)}{\theta\bigl(q^{\widetilde{\alpha}_{i^*}(\xi)},
a_iq^{-\alpha_i(z)};q_i\bigr)},\\
m_{s_{i^*},e}^{s_i}(z,\xi)&=q^{\frac{2\kappa_i}{\mu_i}(\alpha_i(z)-
\widetilde{\alpha}_{i^*}(\xi))}
\frac{\theta\bigl(a_iq^{-\widetilde{\alpha}_{i^*}(\xi)},q^{-\alpha_i(z)};q_i\bigr)}
{\theta\bigl(a_iq^{-\alpha_i(z)},q^{-\widetilde{\alpha}_{i^*}(\xi)};q_i\bigr)},
\end{split}
\end{equation*}
where $\mu_i:=\mu_{\alpha_i}$, $q_i:=q_{\alpha_i}$, $\kappa_i:=\kappa_{\alpha_i}$
and $a_i:=a_{\alpha_i}$.
\end{prop}
Such a simplification of $m_{e,e}^{s_i}$ is 
apparently not possible for the continuous
$q$-Jacobi cases and the Askey-Wilson case.

Recall that the connection coefficients are also the transition functions
for the asymptotically free meromorphic solutions of the quantum
KZ equations \eqref{qaKZ}. Hence the explicit computation of the
connection coefficients may be seen as a quantum analog of
the explicit computation of 
the monodromy of the trigonometric KZ equations from \cite{CAKZ,CInt,O}.

The cocycle property \eqref{CP} of the connection cocycle $M(z,\xi)$
comprises root system analogs of dynamical Yang-Baxter type equations.
We clarify this point of view for the important special example  
$D=(R_0,\Delta_0,t,Q,Q)$ with $R_0$ of type $B_n$
($n\geq 3$). 
Choose the ordering of the basis $\Delta_0$ in such a way that
the braid relations of the associated simple reflections $s_i$
are given by
\begin{equation*}
\begin{split}
s_is_{i+1}s_i&=s_{i+1}s_is_{i+1},\qquad 1\leq i\leq n-2,\\
s_{n-1}s_ns_{n-1}s_n&=s_ns_{n-1}s_ns_{n-1},\\
s_is_j&=s_js_i,\qquad\qquad\,\,\, |i-j|>1.
\end{split}
\end{equation*}
The cocycle condition \eqref{CP} of $M(z,\xi)$ then
yields 
\begin{equation}\label{qdYB}
\begin{split}
M^{s_i}(z,\xi)M^{s_{i+1}}(s_iz,\xi)&M^{s_i}(s_{i+1}s_iz,\xi)=\\
&=M^{s_{i+1}}(z,\xi)M^{s_i}(s_{i+1}z,\xi)M^{s_{i+1}}(s_is_{i+1}z,\xi)
\end{split}
\end{equation}
for $1\leq i\leq n-2$ and 
\begin{equation}\label{qdRE}
\begin{split}
M^{s_{n-1}}(z,\xi)&M^{s_n}(s_{n-1}z,\xi)M^{s_{n-1}}(s_ns_{n-1}z,\xi)
M^{s_n}(s_{n-1}s_ns_{n-1}z,\xi)=\\
=&M^{s_n}(z,\xi)M^{s_{n-1}}(s_nz,\xi)M^{s_n}(s_{n-1}s_nz,\xi)
M^{s_{n-1}}(s_ns_{n-1}s_nz,\xi).
\end{split}
\end{equation}

It is natural to view the equations \eqref{qdYB}
and \eqref{qdRE} as modifications of the
dynamical Yang-Baxter
equation \cite{GN,Fe,EV} and the dynamical reflection equation
\cite{BPO,AR,BFKZ,KM} respectively, with $z$ playing the role
of spectral parameter and $\xi$ the role of dynamical parameter.
This viewpoint 
can be understood 
from the interpretation of the connection
coefficients as the transition functions for asymptotically
free meromorphic solutions of the quantum KZ equations \eqref{qaKZ}.
For example, for the $\textup{GL}_{n+1}$ root system datum, 
the quantum KZ equations
\eqref{qaKZ} coincide with Frenkel-Reshetikhin-Smirnov type quantum
KZ equations. The corresponding transition functions 
are governed by 
the elliptic solution \cite{JKMO,JMO} of the star triangle equation
associated to the integrable $A_n^{(1)}$ face model (see
(\cite[\S 6 \& \S 7]{FR} and \cite{TV1,dM,Ko}) which, in turn,
is known \cite{Fe} 
to be equivalent to Felder's \cite[Prop. 1]{Fe} elliptic solution of the 
dynamical Yang-Baxter equation. Note also the resemblance of  
the explicit expression \cite[Prop. 1]{Fe} of Felder's 
elliptic solution of the dynamical Yang-Baxter equation with the
explicit expression of the connection matrix $M^{s_i}(z,\xi)$ from Proposition
\ref{cUScase}.

Theorem \ref{THM4} now also  
provides an explicit expression of the solution $M^{s_n}(z,\xi)$
of the associated modified dynamical reflection equation
\eqref{qdRE}. 
Note that it depends
on four free parameters
(the Askey-Wilson parameters associated to the short simple root
$\alpha_n$). 
It is expected to provide elliptic dynamical $K$-matrices for
the $A_n^{(1)}$ face models, hence giving rise to
new families of integrable $A_n^{(1)}$ face type models
with reflecting boundary conditions. Thus far, 
$A_n^{(1)}$ face models with reflecting boundary conditions have
only been
constructed for the eight vertex solid-on-solid model \cite{Fi}
with respect to a diagonal solution of the dynamical reflection equation. 

\subsection{Quantum $c$-functions and basic hypergeometric functions}
\label{csection}

The basic hypergeometric function $\mathcal{E}_{sph}(z,\xi)$
is a distinguished $W_0\times W_0$-invariant,
meromorphic solution of the bispectral problem of the RMKC operators 
in case that the root system datum is of the form 
$D=(R_0,\Delta_0,t,\Lambda,\Lambda)$ and $\kappa_a>0$. It was constructed
in \cite{CMehta,CWhit,StSM,StSph} as reproducing kernel of a difference
Fourier transform. It admits an explicit series expansion
in symmetric Macdonald-Koornwinder polynomials. Just as for the
basic Harish-Chandra series, specializing $\xi$ to a polynomial
spectral point turns $\mathcal{E}_{sph}(z,\xi)$ into the
pertinent selfdual Macdonald-Koornwinder polynomial,
see, e.g., \cite[Thm. 2.20]{StSph}.  If $R_0$ has rank one then 
$\mathcal{E}_{sph}$ arises as quantum spherical function on 
noncompact quantum groups, see \cite{KS0}. 

The basic hypergeometric
function is the natural 
basic hypergeometric analog of the Heckman-Opdam \cite{HO} hypergeometric
function associated to root systems, cf. \cite[Thm. 4.4]{CWhit} and
\cite{StSph}. 
The Heckman-Opdam hypergeometric function 
is defined in a completely different fashion, see \cite{HO} and
\cite[Part I]{HS}; the explicit computation of the monodromy representation
of the system of hypergeometric differential equations is used to 
define the Heckman-Opdam hypergeometric function as the explicit
expansion in Harish-Chandra series 
which is fixed under the monodromy representation.
The explicit expansion is the
$c$-function expansion of the Heckman-Opdam hypergeometric function;
the coefficients are expressed
in terms of the Harish-Chandra $c$-function \cite[Def. 6.4]{HO}.

In the current basic hypergeometric context, the analog of the 
$c$-function expansion of the 
basic hypergeometric function $\mathcal{E}_{sph}(z,\xi)$ 
has been derived in \cite{StSph} using the asymptotic analysis of the
basic hypergeometric function from \cite{CWhit}.
We now recall the explicit $c$-function expansion of $\mathcal{E}_{sph}(z,\xi)$
from \cite{StSph} and relate it to the explicit computation of the
connection cocycle.

{}From \cite[Thm. 4.6]{StSph} the $c$-function expansion of 
$\mathcal{E}_{sph}(z,\xi)$ is 
\begin{equation}\label{cfunctionexpansion}
\mathcal{E}_{sph}(z,\xi)=\sum_{w\in W_0}\mathfrak{c}_{sph}(z,w\xi)
\Phi(z,w\xi)
\end{equation}
with the quantum analog $\mathfrak{c}_{sph}\in\mathcal{F}$ of the
Harish-Chandra $c$-function defined as follows.

The higher rank theta function is the holomorphic function 
\[
\vartheta_\Lambda(z):=\sum_{\lambda\in\Lambda}q^{\frac{|\lambda|^2}{2}}
q^{(\lambda,z)}
\]
in $z\in E_{\mathbb{C}}$.
Write $\delta_s^\vee:=\frac{1}{2}\sum_{\alpha\in R_{0,s}^+}\alpha^\vee$,
where $R_{0,s}^+\subset R_0^+$ is the subset of positive short roots
in $R_0$. If $R_0$ is simply laced then we set $R_{0,s}^+=R_0^+$.

\begin{defi}\label{quantumc}
The quantum $c$-function $\mathfrak{c}_{sph}\in\mathcal{F}$ 
associated to the root system datum $D=(R_0,\Delta_0,t,\Lambda,\Lambda)$
is 
\begin{equation*}
\begin{split}
\mathfrak{c}_{sph}(z,\xi):=&
\mathcal{W}(z,\xi)^{-1}\frac{\vartheta_\Lambda\bigl(
\rho+(\kappa_{2a_0}-\kappa_{0})\delta_s^\vee+z+w_0\xi\bigr)}
{\vartheta_\Lambda\bigl((\kappa_{2a_0}-\kappa_{0})\delta_s^\vee+z\bigr)
\vartheta_\Lambda\bigl((\kappa_{2a_0}-\kappa_{2\psi})\delta_s^\vee-\xi)}\\
\times&\prod_{\alpha\in R_0^+}
\frac{\theta\bigl(\widetilde{a}_\alpha q^{\alpha(\xi)},
\widetilde{b}_\alpha q^{\alpha(\xi)}, \widetilde{c}_\alpha q^{\alpha(\xi)},
\widetilde{d}_\alpha q^{\alpha(\xi)};q_\alpha^2\bigr)}
{\theta\bigl(q^{2\alpha(\xi)};q_\alpha^2\bigr)}
\end{split}
\end{equation*}
(see Theorem \ref{THM1}(a) for the definition of the plane wave $\mathcal{W}(z,\xi)$).
If furthermore $\bigl(\Lambda,\alpha^\vee\bigr)=\mathbb{Z}$
for all $\alpha\in R_0$, then
\[
\mathfrak{c}_{sph}(z,\xi)=\mathcal{W}(z,\xi)^{-1}
\frac{\vartheta_\Lambda(\rho+z+w_0\xi)}{\vartheta_\Lambda(z)\vartheta_\Lambda(\xi)}
\prod_{\alpha\in R_0^+}\frac{\theta\bigl(q^{2\kappa_\alpha+\alpha(\xi)};q_\alpha)}
{\theta\bigl(q^{\alpha(\xi)};q_\alpha\bigr)}.
\]
\end{defi}
It follows from \cite[Thm. 2.20 {\bf (iv)} \& Thm. 4.6]{StSph}
that for generic multiplicity functions $\kappa$ satisfying $\kappa_a>0$
for all $a\in R(D)$, 
the basic hypergeometric function $\mathcal{E}_{sph}$ satisfies
\[
\mathcal{E}_{sph}(z,\rho)=\frac{1}{\vartheta_\Lambda((\kappa_{2a_0}-\kappa_0)
\delta_s^\vee-\rho)}\prod_{\alpha\in R_0^+}
\frac{\bigl(\widetilde{a}_\alpha q^{\alpha(\rho)},\widetilde{b}_\alpha
q^{\alpha(\rho)},\widetilde{c}_\alpha q^{\alpha(\rho)}, \widetilde{d}_\alpha
q^{\alpha(\rho)};q_\alpha^2\bigr)_{\infty}}
{\bigl(q^{2\alpha(\rho)};q_\alpha^2\bigr)_{\infty}}.
\]
If one renormalizes $\mathcal{E}_{sph}$ such that $\mathcal{E}_{sph}(z,\rho)=1$,
which is the convention used 
in \cite{StSph}, then it becomes selfdual,
see \cite[Thm. 2.20 {\bf (iii)}]{StSph}. In Remark \ref{notationsrelations}
we will precisely match the present notations to the ones used in \cite{StSph}.

For arbitrary initial data $(D,\kappa,q)$, 
if $\mathfrak{c}\in\mathcal{F}$ then the meromorphic function
\begin{equation}\label{cnewexpansion}
\mathcal{E}(z,\xi):=\sum_{w\in W_0}\mathfrak{c}(z,w\xi)\Phi(z,w\xi)
\end{equation}
in $(z,\xi)\in E_{\mathbb{C}}\times E_{\mathbb{C}}$
is $W_0$-invariant in $z$
if and only $\mathfrak{c}\in\mathcal{F}$ satisfies
\begin{equation}\label{relationsc}
\mathfrak{c}(z,\xi)=m_{e,e}^{s_i}(z,\xi)\mathfrak{c}(s_iz,\xi)+
m_{s_{i^*},e}^{s_i}(z,s_{i^*}\xi)\mathfrak{c}(s_iz,s_{i^*}\xi)
\end{equation}
for all $i\in\{1,\ldots,n\}$ in view of Theorem \ref{THM4}. 
In Section \ref{qcsection} we give a direct proof that 
$\mathfrak{c}_{sph}$ indeed satisfies \eqref{relationsc}
if $D=(R_0,\Delta_0,t,\Lambda,\Lambda)$.
It leads to the following
higher rank analog of the addition formula \eqref{Rid} for Jacobi
theta functions.
\begin{prop}\label{addition}
Suppose $D=(R_0,\Delta_0,t,\Lambda,\Lambda)$ and let $i\in\{1,\ldots,n\}$
such that $\bigl(\Lambda,\alpha_i^\vee\bigr)=\mathbb{Z}$.
Then 
\begin{equation}\label{Ridroot}
\begin{split}
\theta\bigl(q^{\widetilde{\alpha}_{i^*}(\xi)},q^{2\kappa_i-\alpha_i(z)};q_i\bigr)
&\vartheta_\Lambda\bigl(\rho+z+w_0\xi
\bigr)=\\
=&\theta\bigl(q^{2\kappa_i},q^{\widetilde{\alpha}_{i^*}(\xi)-\alpha_i(z)};q_i\bigr)
\vartheta_\Lambda\bigl(\rho+s_iz+w_0\xi
\bigr)\\
-&q^{\widetilde{\alpha}_{i^*}(\xi)}\theta\bigl(q^{2\kappa_i-
\widetilde{\alpha}_{i^*}(\xi)},q^{-\alpha_i(z)};q_i\bigr)
\vartheta_\Lambda\bigl(s_i\rho+z+w_0\xi\bigr).
\end{split}
\end{equation}
\end{prop}
The definition of the basic hypergeometric function as reproducing kernel
of a difference Fourier transform is restricted to the twisted equal
lattice case, $D=(R_0,\Delta_0,t,\Lambda,\Lambda)$. With the explicit
expressions of the connection coefficients now available, it is
natural to try to extend the method employed in the
construction of the Heckman-Opdam
hypergeometric function and define the appropriate analog of the basic 
hypergeometric function
for all root system data $D$ as the expansion \eqref{cnewexpansion}
for a distinguished solution $\mathfrak{c}\in\mathcal{F}$
of the equations \eqref{relationsc}.
This is a subtle
matter, since the equations \eqref{relationsc} do not determine
$\mathfrak{c}\in\mathcal{F}$ uniquely, cf. Subsection \ref{ReflSect}.
We do not pursue this issue in the present paper, although we will make some
initial steps in the analysis of the equations \eqref{relationsc} in 
Section \ref{qcsection}.

\subsection{Reflectionless basic Harish-Chandra series}
\label{ReflSect}
In, e.g., \cite{RuR1,RuR2,vDP}, reflectionless analytic difference
operators are studied. These are difference analogs of 
one-dimensional Schr{\"o}dinger operators
admitting a meromorphic eigenfunction $\phi(\cdot,p)$ 
in $(\cdot,p)\in\mathbb{C}^2$ with eigenvalue
$e^p+e^{-p}$ and having 
plane wave asymptotics 
\begin{equation*}
\begin{split}
\phi(x,p)&\sim e^{\sqrt{-1}xp},\qquad\qquad\qquad\qquad
\qquad\quad\Re(x)\rightarrow\infty,\\
\phi(x,p)&\sim \alpha(p)e^{\sqrt{-1}xp}+\beta(p)e^{-\sqrt{-1}xp},
\qquad\, \Re(x)\rightarrow-\infty
\end{split}
\end{equation*}
with $\beta\equiv 0$. 
This has the following analog
for RMKC operators.

Write
$\ell(\nu):=\textup{min}\bigl((\nu,\widetilde{\alpha}_i^\vee)\bigr)_{i=1}^n$
for $\nu\in\widetilde{\Lambda}$. 
Then for generic $\xi\in E_{\mathbb{C}}$,
\[
\Phi(z-\nu,\xi)=\frac{q^{(\widetilde{\rho},\rho-\xi)}\Gamma_0(\xi)}
{\widetilde{\mathcal{S}}(\xi)}e^{-(\rho+w_0\xi,z-\nu)}(1+
\mathcal{O}(q^{\ell(\nu)}))
\]
as $\ell(\nu)\rightarrow\infty$, uniformly for $z$ in compacta of
$E_{\mathbb{C}}$, in view of Theorem \ref{THM1}. It describes the 
plane wave asymptotics of $\Phi(z,\xi)$ as common eigenfunction of the
RMKC operators for
the real part $\Re(z)$ of $z$ deep in the negative fundamental Weyl chamber
$E_-:=\{v\in E \,\, | \,\, \alpha(v)<0 \quad \forall\, \alpha\in R_0^+\}$. 
By Theorem \ref{THM3}, $\Phi(z,\xi)$ has plane wave asymptotics
for $\Re(z)$ deep in an arbitrary Weyl chamber $\tau(E_-)$ ($\tau\in W_0$). 
In particular, for the Weyl chambers neighboring $E_-$, 
\begin{equation*}
\begin{split}
\Phi(z-s_i\nu,\xi)&\sim m_{e,e}^{s_i}(s_iz,\xi)
\frac{q^{(\widetilde{\rho},\rho-\xi)}\Gamma_0(\xi)}{\widetilde{\mathcal{S}}(\xi)}
e^{-(s_i\rho+w_0s_{i^*}\xi,z-s_i\nu)}\\
&+m_{s_{i^*},e}^{s_i}(s_iz,\xi)\frac{q^{(\widetilde{\rho},\rho-s_{i^*}\xi)}
\Gamma_0(s_{i^*}\xi)}{\widetilde{\mathcal{S}}(s_{i^*}\xi)}
e^{-(s_i\rho+w_0\xi,z-s_i\nu)}
\end{split}
\end{equation*}
as $\ell(\nu)\rightarrow\infty$ for generic
$z,\xi\in E_{\mathbb{C}}$ ($i=1,\ldots,n$) by Theorem \ref{THM4}.
Thus we come to the following analog of reflectionless
in the context of RMKC operators.
\begin{defi}
We say that the RMKC operators are reflectionless 
if $m_{e,e}^{s_i}\equiv 0$ for $i=1,\ldots,n$. In this case 
the associated basic Harish-Chandra series $\Phi$ is said to be
reflectionless.
\end{defi}
By Theorem \ref{THM4}, the RMKC operators are reflectionless if and only if
\begin{equation}\label{equive}
\mathfrak{e}_\alpha(x,y)=\widetilde{\mathfrak{e}}_\alpha(y,x)
\end{equation}
as meromorphic functions in $(x,y)\in\mathbb{C}\times\mathbb{C}$ for all
$\alpha\in R_0$.
The following result now follows by straightforward 
computations using the functional equation \eqref{fe} for the normalized
Jacobi theta function $\theta(x;q)$ 
(cf. \cite[Prop. 3.1]{StQ} for a weaker statement).
\begin{prop}\label{ReflectionlessProp}
Suppose that the multiplicity function $\kappa$
satisfies
\begin{equation}\label{BAconditions}
\begin{split}
&\kappa_{\alpha}\pm\kappa_{\alpha^{(1)}},
\kappa_{2\alpha}\pm\kappa_{2\alpha^{(1)}},
\kappa_\alpha\pm\kappa_{2\alpha},
\kappa_{\alpha^{(1)}}\pm\kappa_{2\alpha^{(1)}}\in\mu_\alpha\mathbb{Z},\\
&\kappa_{\alpha}+\kappa_{2\alpha}+\kappa_{\alpha^{(1)}}+\kappa_{2\alpha^{(1)}}
\in 2\mu_\alpha\mathbb{Z}
\end{split}
\end{equation}
for all $\alpha\in R_0$.
Then the RMKC operators are reflectionless, and $m_{s_{i^*},e}^{s_i}\equiv 1$
for $i=1,\ldots,n$.
\end{prop}
If the multiplicity function $\kappa$ 
satisfies \eqref{BAconditions},
then so does the dual multiplicity function $\widetilde{\kappa}$ (see 
Lemma \ref{dualparameterslem}). Note the following special cases:\\

\noindent
{\bf Continuous $q$-ultraspherical case:} 
$\bigl(\Lambda,\alpha^\vee\bigr)=\mathbb{Z}=\bigl(\widetilde{\Lambda},
\widetilde{\alpha}^\vee\bigr)$. Then \eqref{BAconditions} reduces 
to the single condition
\[
\kappa_\alpha\in\frac{\mu_\alpha}{2}\mathbb{Z}.
\]
{\bf Continuous $q$-Jacobi case:} $\bigl(\Lambda,\alpha^\vee\bigr)=
\mathbb{Z}$ and $\bigl(\widetilde{\Lambda},\widetilde{\alpha}^\vee\bigr)=
2\mathbb{Z}$, then \eqref{BAconditions} reduces to 
\[
\kappa_\alpha,\kappa_{\alpha^{(1)}}\in\frac{\mu_\alpha}{2}\mathbb{Z},
\qquad \kappa_\alpha+\kappa_{\alpha^{(1)}}\in\mu_\alpha\mathbb{Z},
\] 
or $\bigl(\Lambda,\alpha^\vee\bigr)=
2\mathbb{Z}$ and $\bigl(\widetilde{\Lambda},\widetilde{\alpha}^\vee\bigr)=
\mathbb{Z}$, then \eqref{BAconditions} reduces to
\[
\kappa_\alpha,\kappa_{2\alpha}\in\frac{\mu_\alpha}{2}\mathbb{Z},\qquad
\kappa_\alpha+\kappa_{2\alpha}\in\mu_\alpha\mathbb{Z}.
\]
Theorem \ref{THM3} and Proposition \ref{ReflectionlessProp}
give the following result.
\begin{cor}\label{W0invCOR}
If the multiplicity function $\kappa$ satisfies \eqref{BAconditions},
then the reflectionless basic Harish-Chandra series $\Phi$ satisfies
\begin{equation}\label{Phiinv}
\Phi(wz,w_0ww_0\xi)=\Phi(z,\xi)\qquad \forall\, w\in W_0.
\end{equation}
\end{cor}
The analysis of quantum $c$-functions simplifies
in the present context of reflectionless RMKC operators,
since the conditions \eqref{relationsc} for
$i=1,\ldots,n$ for $\mathfrak{c}\in\mathcal{F}$ are equivalent to the 
invariance property
\[
\mathfrak{c}(wz,w_0ww_0\xi)=\mathfrak{c}(z,\xi)\qquad \forall\, w\in W_0
\]
if $\kappa$ satisfies \eqref{BAconditions}.
Consequently, under the assumption \eqref{BAconditions}
on the multiplicity function $\kappa$,
\begin{equation}\label{Phiplus}
\Phi_+(z,\xi):=\sum_{w\in W_0}\Phi(z,w\xi)
\end{equation}
is a $W_0\times W_0$-invariant meromorphic solution
of the bispectral problem of the reflectionless RMKC operators.
Note that 
in the twisted equal lattice
case $D=(R_0,\Delta_0,t,\Lambda,\Lambda)$ with $\kappa$ satisfying
the reflectionless conditions \eqref{BAconditions}, $\Phi_+$
does {\it not} coincide with the basic hypergeometric function
$\mathcal{E}_{sph}$.
\subsection{Multivariable Baker-Akhiezer functions}\label{BAsection}

In \cite{Ch,ChE}, multivariable Baker-Akhiezer functions associated to
RMKC operators are defined under suitable restrictions on the multiplicity
function $\kappa$ for reduced semisimple root data
$D=(R_0,\Delta,\bullet,P,\widetilde{P})$ and for the Koornwinder 
root system datum. 
The conditions \cite[\S 2.1.3]{ChE} 
on the multiplicity function $\kappa$ for the
multivariable Baker-Akhiezer function to be defined 
then read
\begin{enumerate}
\item[(BA1)] if $\alpha\in R_0$ with
$\bigl(\Lambda,\alpha^\vee\bigr)=\mathbb{Z}=\bigl(\widetilde{\Lambda},
\widetilde{\alpha}^\vee\bigr)$ (continuous $q$-ultraspherical case) 
then
$\kappa_\alpha\in\frac{\mu_\alpha}{2}\mathbb{Z}_{\leq 0}$,
\item[(BA2)] if $\alpha\in R_0$ with
$\bigl(\Lambda,\alpha^\vee\bigr)=2\mathbb{Z}=\bigl(\widetilde{\Lambda},
\widetilde{\alpha}^\vee\bigr)$ (Askey-Wilson case) then 
\begin{equation*}
\begin{split}
&\kappa_\alpha\pm\kappa_{2\alpha},\kappa_\alpha\pm\kappa_{\alpha^{(1)}}\in
\mu_\alpha\mathbb{Z}_{\leq 0},\\
&\kappa_{\alpha^{(1)}}\pm\kappa_{2\alpha^{(1)}}, 
\kappa_{2\alpha}\pm\kappa_{2\alpha^{(1)}}\in\mu_\alpha\mathbb{Z}_{<0},\\
&\kappa_\alpha+\kappa_{2\alpha}+\kappa_{\alpha^{(1)}}+\kappa_{2\alpha^{(1)}}\in
2\mu_\alpha\mathbb{Z}
\end{split}
\end{equation*}
\end{enumerate}
(the continuous $q$-Jacobi cases $\bigl(\Lambda,\alpha^\vee\bigr)=\mathbb{Z}$
and $\bigl(\widetilde{\Lambda},\widetilde{\alpha}^\vee\bigr)=2\mathbb{Z}$,
respectively $\bigl(\Lambda,\alpha^\vee\bigr)=2\mathbb{Z}$ and
$\bigl(\widetilde{\Lambda},\widetilde{\alpha}^\vee\bigr)=\mathbb{Z}$,
do not occur for reduced semisimple and Koornwinder root data). 

\begin{rema}\label{pc}
{\bf (i)} For the reduced semisimple root data $D$ the
free parameters $m_\alpha$ in \cite[\S 2.1.1]{ChE} corresponds to 
$-2\kappa_\alpha/\mu_\alpha$. For the Koornwinder root system datum $D$
the free parameters $m_i$ ($1\leq i\leq 5$) in \cite[\S 2.1.2]{ChE}
are related to $\kappa$ by
\begin{equation*}
\begin{split}
&m_1=-\kappa_\alpha-\kappa_{2\alpha},\qquad m_2=-\frac{1}{2}
-\kappa_{\alpha^{(1)}}-\kappa_{2\alpha^{(1)}},\\
&m_3=-\kappa_\alpha+\kappa_{2\alpha},\qquad
m_4=-\frac{1}{2}-\kappa_{\alpha^{(1)}}+\kappa_{2\alpha^{(1)}},\\
&m_5=-2\kappa_\beta
\end{split}
\end{equation*}
where $\alpha\in R_0$ (resp. $\beta\in R_0$) is a short (resp. long)
root and the root system $R_0$ is normalized such that long roots
have squared length two. 
Here the fifth free parameter $m_5$ (resp.
$\kappa_\beta$) should only be taken into account 
if the rank $n$ of $R_0$ is $\geq 2$.\\ 
{\bf (ii)} The results of the previous subsection apply
if $D$ is a reduced semisimple or a Koornwinder root system datum
and the multiplicity function
$\kappa$ satisfies (BA1) and (BA2), since conditions (BA1)
and (BA2) imply the reflectionless
conditions \eqref{BAconditions}.
In particular, the RMKC operators are reflectionless and the 
basic Harish-Chandra series $\Phi$ satisfies the invariance property
\eqref{Phiinv}.
\end{rema}
The following result traces back to \cite[\S 4.4]{LS}. 
We discuss its proof at the end of Subsection \ref{HCsub}.
\begin{prop}\label{PROP5}
Let $(D,\kappa,q)$ be an initial datum with
a reduced semisimple or a Koornwinder root system datum $D$ and with
multiplicity function $\kappa$ satisfying (BA1) and (BA2). 
Let $\psi(\lambda,x)$ be the multivariable Baker-Akhiezer function
associated to $(D,\kappa,q)$ (see \cite[\S 3.1]{ChE}, in particular the
definition below \cite[(3.8)]{ChE}), 
where we use the parameter correspondence
as indicated in Remark \ref{pc}{\bf (i)}. Then 
\[
\Phi(z,\xi)=\textup{cst}\,\psi(-w_0\xi,z)
\]
as meromorphic functions in $(z,\xi)\in E_{\mathbb{C}}\times E_{\mathbb{C}}$
for some constant $\textup{cst}\in\mathbb{C}^*$.
\end{prop}
\begin{rema}
{\bf (i)}
The constant $\textup{cst}$ can easily be explicitly
computed by comparing the
normalizations of $\Phi$ and $\psi$, cf. Subsection \ref{HCsub}.\\
{\bf (ii)} Proposition \ref{PROP5} allows to rederive
various fundamental properties of the
multivariable Baker-Akhiezer function as direct consequences of
the analogous properties of the basic Harish-Chandra series.
For instance, the selfduality \cite[Thm. 3.3(iii)]{ChE}
of the multivariable Baker-Akhiezer function $\psi$ becomes
a special case of the
selfduality of the basic Harish-Chandra series $\Phi$ (Theorem
\ref{THM2}), while the $W_0$-invariance \cite[Lem 3.4(i)]{ChE}
of $\psi$ is a special case of the $W_0$-invariance of the
reflectionless basic Harish-Chandra series (Corollary \ref{W0invCOR}).
\end{rema}
Proposition \ref{PROP5} opens the way to
study the results \cite{Ch,ChE} on multivariable 
Baker-Akhiezer functions on the level of 
(reflectionless) basic Harish-Chandra series.
In particular, one can now study the extra symmetries 
\cite[(3.4)-(3.6)]{ChE} and the terminating series expansion 
property \cite[(3.3)]{ChE} of the multivariable
Baker-Akhiezer function $\psi$ on the level of (reflectionless)
basic Harish-Chandra series.
For instance,
the terminating series expansion of $\psi$ becomes
the following surprising property of the
expansion coefficients $\widehat{\Gamma}_\alpha(\xi)$ \eqref{EC}
of the basic Harish-Chandra series. 
\begin{cor}
Let $(D,\kappa,q)$ be an initial datum with
$D$ a reduced semisimple root system datum and $\kappa$ a 
multiplicity function satisfying (BA1). 
Then
\[
\widehat{\Gamma}_\alpha(\xi)=0
\]
as meromorphic function in $\xi\in E_{\mathbb{C}}$ unless 
$\alpha\in Q^+$ is of the form 
$\alpha=\frac{1}{2}\sum_{\beta\in R_0^+}l_\beta\beta$ with
$0\leq l_\beta\leq -4\kappa_\beta/\mu_\beta$ for all 
$\beta\in R_0^+$. 
\end{cor}

\section{Notations}\label{id}
We continue the introduction of basic notations as started in
Subsection \ref{idsection}. We refer to \cite{StB} for further details.
\subsection{The affine root system}
For fixed root system datum 
$D=(R_0,\Delta_0,\bullet,\Lambda,\widetilde{\Lambda})$
with ambient Euclidean space $E$ let $\widehat{E}$ be the linear space
of real affine linear functions on $E$. Then 
$\widehat{E}\simeq\mathbb{R}c\oplus E$, where 
$a=\eta c+v$ with $\eta\in\mathbb{R}$ and $v\in E$ is
interpreted as the affine linear function $v^\prime\mapsto\eta+(v,v^\prime)$.

Let $V$ be the real span of the
roots. We view the linear space 
$\widehat{V}$ of real affine linear
functions on $V$ as the subspace of $\widehat{E}$ which are constant on the 
orthocomplement $V^\perp$ of $V$ in $E$. 

The extended affine Weyl group $W=W_0\ltimes\widetilde{\Lambda}$ 
acts on $E$ and on its complexification $E_{\mathbb{C}}$ with the canonical
action of $W_0$ and with $\widetilde{\Lambda}$ acting
by translations $\tau(\nu)z:=z+\nu$ 
($\nu\in\widetilde{\Lambda}$). It
induces a linear $W$-action on $\widehat{E}$. Note that
\[
\tau(\nu)\alpha^{(r)}=\alpha^{(r-(\nu,
\widetilde{\alpha}^\vee))}
\]
for $\alpha\in R_0$, $\nu\in\widetilde{\Lambda}$ 
and $r\in\mathbb{Z}$, hence $R^\bullet$ and $R$
are $W$-invariant.

For an affine root $\alpha^{(r)}\in R^\bullet\subset\widehat{V}$ let 
$s_{\alpha^{(r)}}\in W$ be the orthogonal reflection in the affine hyperplane
$\textup{ker}(\alpha^{(r)})$. Then $s_{\alpha^{(r)}}=\tau(-r\widetilde{\alpha})
s_\alpha$, with $s_\alpha\in W_0$ the orthogonal reflection in the hyperplane
$\alpha^\perp$. We write $s_i:=s_{\alpha_i}\in W$ ($0\leq i\leq n$) for the
simple reflections. They generate the affine Weyl subgroup 
$W^\bullet:=W_0\ltimes \widetilde{Q}$ of $W$. Note that $s_0=
\tau(\widetilde{\psi})s_\psi$.

Let $R^{\bullet,+}$ and $R^{\bullet,-}$ be the positive and negative affine roots
of $R^\bullet$ with respect to the 
basis $\Delta:=(\alpha_0,\alpha_1,\ldots,\alpha_n)$ of $R^\bullet$.
The length of $w\in W$ is defined by
\[
l(w):=\#\bigl(R^{\bullet,+}\cap w^{-1}R^{\bullet,-}\bigr),\qquad w\in W.
\]
We have $W\simeq\Omega\ltimes W^\bullet$ with $\Omega=\Omega(D)$ the subgroup
\[
\Omega:=\{w\in W \,\, | \,\, l(w)=0\}.
\]
For $\nu\in\widetilde{\Lambda}$ let $u(\nu)\in W$ be the element of minimal
length in $\tau(\nu)W_0$ and write $v(\nu):=u(\nu)^{-1}\tau(\nu)\in W_0$. Then
\[
\Omega=\{u(\nu)\}_{\nu\in\widetilde{\Lambda}^+_{min}}
\]
with $\widetilde{\Lambda}^+_{min}$ the set of dominant minimal weights in
$\widetilde{\Lambda}$,
\[
\widetilde{\Lambda}^+_{min}:=\{\nu\in\widetilde{\Lambda}\,\, |\,\, 
\bigl(\nu,\widetilde{\alpha}^\vee\bigr)\in\{0,1\}\quad\forall\, 
\alpha\in R_0^+\}.
\]
The set of dominant weights in $\widetilde{\Lambda}$ is
\[
\widetilde{\Lambda}^+:=\{\nu\in\widetilde{\Lambda} \,\, | \,\,
\bigl(\nu,\widetilde{\alpha}^\vee\bigr)\geq 0\quad \forall\, \alpha\in
R_0^+\}.
\]
\subsection{The dual affine root system}\label{dars}
The root system datum $\widetilde{D}=(\widetilde{R}_0,
\widetilde{\Delta}_0,\bullet,\widetilde{\Lambda},\Lambda)$ dual to $D$
gives rise to a dual reduced affine root system
\[
\widetilde{R}^\bullet=\{\widetilde{\alpha}^{(r)}=\mu_{\widetilde{\alpha}}rc+
\widetilde{\alpha}\}_{\alpha\in R_0, r\in\mathbb{Z}}
\]
and its extension $\widetilde{R}:=R(\widetilde{D})$. 
The associated extended affine
Weyl group is $\widetilde{W}:=W_0\ltimes\Lambda$. The additional
simple affine root of $\widetilde{R}$
is denoted by $\widetilde{\alpha}_0$. Write
$\widetilde{s}_i:=s_{\widetilde{\alpha}_i}$ for $i\in\{0,\ldots,n\}$.
Note that $\widetilde{s}_i=s_i$ for $1\leq i\leq n$ while
$\widetilde{s}_0=\tau(\theta)s_\theta$ with $\theta\in R_0^+$ the
highest short root, since $\widetilde{\alpha}_0=
\mu_{\widetilde{\theta}}c-\widetilde{\theta}$.
We write $\widetilde{\Omega}:=\Omega(\widetilde{D})$ 
so that $\widetilde{W}=\widetilde{\Omega}
\ltimes\widetilde{W}^\bullet$ with $\widetilde{W}^\bullet=W_0\ltimes Q$
the affine Weyl group associated to $\widetilde{R}$.

For $\lambda\in\Lambda$ we write $\widetilde{u}(\lambda)\in\widetilde{W}$ 
for the shortest element in $\tau(\lambda)W_0$ and 
$\widetilde{v}(\lambda):= \widetilde{u}(\lambda)^{-1}\tau(\lambda)\in W_0$.
The set of dominant weights in $\Lambda$ is
\[\Lambda^+:=\{\lambda\in\Lambda
\,\, | \,\, \bigl(\lambda,\alpha^\vee\bigr)
\geq 0\quad \forall\, \alpha\in R_0^+\}
\]
and the set of dominant mimiscule weights in $\Lambda$ is
\[
\Lambda_{min}^+:=\{\lambda\in\Lambda\,\, | \,\, 
\bigl(\lambda,\alpha^\vee\bigr)\in\{0,1\}\quad \forall\,\alpha\in R_0^+ \}.
\]
As in the previous subsection, we have
$\widetilde{\Omega}=\{\widetilde{u}(\lambda)\}_{\lambda\in
\Lambda_{min}^+}$.
\subsection{Multiplicity functions}\label{mf}
Let $\mathcal{M}(D)$ be the space of $W$-invariant functions
$\kappa: R(D)\rightarrow \mathbb{R}$. Its value at $a\in R$ is denoted
by $\kappa_a$. Recall the convention
$\kappa_{2\alpha^{(r)}}:=\kappa_{\alpha^{(r)}}$ if $2\alpha^{(r)}\not\in R$
(i.e. if $\bigl(\Lambda,\alpha^\vee\bigr)=\mathbb{Z}$). The involution
$D\mapsto \widetilde{D}$ on root system data
extends to multiplicity functions as follows
(see \cite{Ha,StB}).

\begin{lem}\label{dualparameterslem}
There exists a unique linear isomorphism $\mathcal{M}(D)\overset{\sim}
{\longrightarrow}
\mathcal{M}(\widetilde{D})$, $\kappa\mapsto\widetilde{\kappa}$, satisfying
$\widetilde{\widetilde{\kappa}}=\kappa$ and satisfying
\[
\widetilde{\kappa}_{\widetilde{\alpha}^{(1)}}=\kappa_{2\alpha},
\qquad 
\widetilde{\kappa}_{\widetilde{\alpha}}=\kappa_{\alpha},
\qquad
\widetilde{\kappa}_{2\widetilde{\alpha}^{(1)}}=\kappa_{2\alpha^{(1)}}
\]
for all $\alpha\in R_0$.
\end{lem}
Recall from Subsection \ref{idsection} that we associated to 
the initial datum $(D,\kappa,q)$ Askey-Wilson (AW) parameters
$a_\alpha=a_\alpha(D,\kappa,q),\ldots,d_\alpha=d_\alpha(D,\kappa,q)$
for all $\alpha\in R_0$, as well
as dual AW parameters \eqref{dualAW}. 
Then for all $\alpha\in R_0$,
\[
\widetilde{a}_\alpha=a_{\widetilde{\alpha}}(\widetilde{D},\widetilde{\kappa},q),
\ldots,\widetilde{d}_\alpha=d_{\widetilde{\alpha}}(\widetilde{D},
\widetilde{\kappa},q)
\]
since $\mu_{\widetilde{\alpha}}=\mu_\alpha$.

\section{The difference integrable equations and
asymptotic analysis}\label{bHCs}
Basic Harish-Chandra series are meromorphic
common eigenfunctions of the RMKC operators, characterized by suitably
asymptotically free behaviour deep in an appropriate Weyl chamber.
Basic Harish-Chandra series have been considered in 
various different contexts \cite{CQKZ2,EK,KK,LS,vMS,vM,StSph,NS}. 
In \cite{CQKZ2} their existence
was predicted based on the correspondence with solutions of quantum KZ 
equations. In \cite{EK,KK} the basic Harish-Chandra series were considered
for $R_0$ of type $A_{n-1}$ using vertex operators. In \cite{LS} basic
Harish-Chandra series were constructed as formal power series using classical
methods from harmonic analysis. In the series of papers
\cite{vMS,vM,StSph} Cherednik's prediction was worked out in detail
for the twisted case $\bullet=t$ and with $\widetilde{\Lambda}=\Lambda$
by relating the basic Harish-Chandra series to 
asymptotically free solutions of (bispectral extensions)
of quantum KZ equations through the difference Cherednik-Matsuo correspondence
\cite{CQKZ2,CInd,Ka,StI}. In \cite{NS} a direct approach is undertaken
to derive the fundamental properties of the basic Harish-Chandra series
when $R_0$ is of type $A_n$. In this section we shortly discuss the
extension of the methods from \cite{vMS,vM,StSph} to the present 
context, which includes the untwisted theory and has extra freedom in
the choice of lattices. We only give the proof if it needs
new arguments compared to the twisted case $\bullet=t$ with 
$\widetilde{\Lambda}=\Lambda$. So throughout this section 
$(D,\kappa,q)$ stands for an arbitrary choice of initial datum
unless explicitly specified otherwise.

\subsection{Bispectral quantum KZ equations}\label{bqKZsection}

Define for $a\in R^\bullet$ the meromorphic function
$c_a(\cdot)=c_a(\cdot;D,\kappa,q)$ on $E_{\mathbb{C}}$ by
\[
c_a(z):=\frac{(1-q^{\kappa_a+\kappa_{2a}+a(z)})(1+q^{\kappa_a-\kappa_{2a}+a(z)})}
{1-q^{2a(z)}}.
\]
We write 
$c_i(\cdot;\kappa,q)=c_{\alpha_i}(\cdot;D,\kappa,q)$
and $\widetilde{c}_i(\cdot;\widetilde{\kappa},q)=
c_{\widetilde{\alpha}_i}(\cdot;\widetilde{D},\widetilde{\kappa},q)$
for $i\in\{0,\ldots,n\}$.

Let $\mathcal{M}$ be the field of meromorphic function on $E_{\mathbb{C}}\times
E_{\mathbb{C}}$. The contragredient actions of $W$ and $\widetilde{W}$
on $E_{\mathbb{C}}$ give rise to an action of $W\times \widetilde{W}$
on $\mathcal{M}$ by field automorphisms. Note that
$\mathcal{F}=\mathcal{M}^{\tau(\widetilde{\Lambda})\times\tau(\Lambda)}$.

Let $\mathcal{M}\otimes_{\mathbb{C}}\textup{End}_{\mathbb{C}}
(\mathcal{V})\simeq\textup{End}_{\mathcal{M}}(\mathcal{M}\otimes_{\mathbb{C}}
\mathcal{V})$ be the space of $\textup{End}_{\mathbb{C}}(\mathcal{V})$-valued
meromorphic functions on $E_{\mathbb{C}}\times E_{\mathbb{C}}$, where
$\mathcal{V}:=\bigoplus_{\sigma\in W_0}\mathbb{C}v_\sigma$. Let 
$\chi: R_0\rightarrow \{0,1\}$ be the characteristic function of
$R_0^-$ in $R_0$.

\begin{thm}\label{BqKZ}
There exists unique 
$C_{(w,w^\prime)}\in \textup{End}_{\mathcal{M}}(\mathcal{M}\otimes_{\mathbb{C}}
\mathcal{V})$ ($(w,\widetilde{w})\in W\times\widetilde{W}$)
satisfying 
the cocycle conditions
\begin{equation}\label{cocC}
C_{(vw,\widetilde{v}\widetilde{w})}(z,\xi)=
C_{(v,\widetilde{v})}(z,\xi)C_{(w,\widetilde{w})}(v^{-1}z,\widetilde{v}^{-1}\xi),
\qquad \forall\, (v,\widetilde{v}), (w,\widetilde{w})\in W\times\widetilde{W}
\end{equation}
and $C_{(e,e)}(z,\xi)=\textup{Id}_{\mathcal{V}}$,
and satisfying for all $\sigma\in W_0$,
\begin{equation}\label{C}
\begin{split}
C_{(s_0,e)}(z,\xi)v_\sigma&=\frac{q^{\widetilde{\psi}(\sigma\xi)}v_{s_{\psi}\sigma}}
{q^{\kappa_0}c_0(z;-\kappa,q)}+
\left(\frac{c_0(z;-\kappa,q)-q^{-2\chi(\sigma^{-1}\psi)\kappa_0}}
{c_0(z;-\kappa,q)}\right)v_\sigma,\\
C_{(s_i,e)}(z,\xi)v_\sigma&=
\frac{v_{s_i\sigma}}{q^{\kappa_i}c_i(z;-\kappa,q)}+
\left(\frac{c_i(z;-\kappa,q)-q^{-2\chi(-\sigma^{-1}\alpha_i)\kappa_i}}
{c_i(z;-\kappa,q)}\right)v_\sigma,\\
C_{(u(\nu),e)}(z,\xi)v_\sigma&=q^{-(w_0\nu,\sigma\xi)}v_{v(\nu)^{-1}\sigma}
\end{split}
\end{equation}
for $1\leq i\leq n$ and $\nu\in\widetilde{\Lambda}_{min}^{+}$
and 
\begin{equation}\label{Cdual}
\begin{split}
C_{(e,\widetilde{s}_0)}(z,\xi)v_\sigma&=\frac{q^{\theta(\sigma^{-1}z)}
v_{\sigma s_{\theta}}}{q^{\widetilde{\kappa}_0}
\widetilde{c}_0(\xi;-\widetilde{\kappa},q)}+
\left(\frac{\widetilde{c}_0(\xi;-\widetilde{\kappa},q)-
q^{-2\chi(\sigma\theta)\widetilde{\kappa}_0}}
{\widetilde{c}_0(\xi;-\widetilde{\kappa},q)}\right)
v_\sigma,\\
C_{(e,\widetilde{s}_i)}(z,\xi)v_\sigma&=\frac{v_{\sigma s_i}}
{q^{\widetilde{\kappa}_i}\widetilde{c}_i(\xi;-\widetilde{\kappa},q)}+
\left(\frac{\widetilde{c}_i(\xi;-\widetilde{\kappa},q)-
q^{-2\chi(-\sigma\alpha_i)\widetilde{\kappa}_i}}
{\widetilde{c}_i(\xi;-\widetilde{\kappa},q)}\right)
v_\sigma,\\
C_{(e,\widetilde{u}(\lambda))}(z,\xi)v_\sigma&=
q^{-(w_0\lambda,\sigma^{-1}z)}v_{\sigma \widetilde{v}(\lambda)}
\end{split}
\end{equation}
for $1\leq i\leq n$ and $\lambda\in\Lambda_{min}^+$. 
\end{thm}
\begin{proof}
The proof in \cite{vMS,vM,StSph} in the twisted equal lattice case
$D=(R_0,\Delta_0,t,\Lambda,\Lambda)$
uses the affine intertwiners and the duality antiisomorphism
of the double affine Hecke algebra. 
This proof easily extends to the present setup (for the double affine Hecke
algebra in the present context, see \cite{StB}).
\end{proof}
By Theorem \ref{BqKZ},
\begin{equation}\label{nablaaction}
\bigl(\nabla(w,\widetilde{w})f\bigr)(z,\xi):=C_{(w,\widetilde{w})}(z,\xi)
f(w^{-1}z,\widetilde{w}^{-1}\xi)
\end{equation}
defines a complex linear left
action $\nabla=\nabla(D,\kappa,q)$
of $W\times\widetilde{W}$ on $\mathcal{M}\otimes_{\mathbb{C}}\mathcal{V}$.
Following \cite{vMS,vM,StSph}, we arrive now at the definition of the bispectral
quantum Khnizhnik-Zamolodchikov (KZ) equations.
\begin{defi}
We say that $f\in\mathcal{M}\otimes_{\mathbb{C}}\mathcal{V}$ is a meromorphic
solution of the bispectral quantum KZ equations 
if 
\begin{equation}\label{BqKZdef}
\nabla(\tau(\nu),\tau(\lambda))f=f\qquad \forall\, \nu\in\widetilde{\Lambda},
\,\,\forall\, \lambda\in\Lambda.
\end{equation}
We write $\textup{Sol}_{KZ}=\textup{Sol}_{KZ}(D,\kappa,q)$ for the 
vector space over $\mathcal{F}$ of meromorphic $\mathcal{V}$-valued functions
$f\in\mathcal{M}\otimes_{\mathbb{C}}\mathcal{V}$ satisfying the bispectral
quantum KZ equations \eqref{BqKZdef}.
\end{defi}
Note that
\[
\bigl(\nabla(\tau(\nu),\tau(\lambda))f\bigr)(z,\xi)=
C_{(\tau(\nu),\tau(\lambda))}(z,\xi)f(z-\nu,\xi-\lambda),\qquad
\nu\in\widetilde{\Lambda},\,\, \lambda\in\Lambda,
\]
hence the bispectral quantum KZ equations form a compatible system of
linear
difference equations (an integrable difference connection). The solution
space $\textup{Sol}_{KZ}$ of the bispectral quantum KZ equations is a
$W_0\times W_0$-invariant complex linear
subspace of $\mathcal{M}\otimes_{\mathbb{C}}
\mathcal{V}$ with respect to the action 
$\nabla|_{W_0\times W_0}$ of $W_0\times W_0$ on 
$\mathcal{M}\otimes_{\mathbb{C}}\mathcal{V}$.

Note that the coefficients
$C_{\tau(\nu),\tau(\lambda)}(z,\xi)$ are in fact rational functions 
in 
\[
(q^z,q^\xi)\in \mathbb{T}:=
\textup{Hom}\bigl(\Lambda\times\widetilde{\Lambda},
\mathbb{C}^*\bigr),
\]
where we interpret $q^z\in\textup{Hom}\bigl(\Lambda,\mathbb{C}^*\bigr)$
and $q^\xi\in\textup{Hom}\bigl(\widetilde{\Lambda},\mathbb{C}^*\bigr)$
as $\lambda\mapsto q^{(\lambda,z)}$ ($\lambda\in\Lambda$)
and $\nu\mapsto q^{(\nu,\xi)}$ ($\nu\in\widetilde{\Lambda}$) respectively.
Hence the bispectral quantum KZ equations, restricted
to meromorphic $\mathcal{V}$-valued functions on $\mathbb{T}$, form
a compatible system of $q$-difference equations (an integrable
$q$-connection). It is in this form that quantum KZ type 
equations usually appear, see, e.g., 
\cite{FR,CQKZ,CQKZ2,vMS,vM,vMS} and references therein.

\begin{rema}[Duality symmetry]\label{DS}
Let $j: \mathcal{M}\otimes_{\mathbb{C}}\mathcal{V}\rightarrow
\mathcal{M}\otimes_{\mathbb{C}}\mathcal{V}$ be the complex linear map
defined by
\[
j\bigl(f(\cdot,\cdot)v_\sigma\bigr):=\widetilde{f}(\cdot,\cdot)v_{\sigma^{-1}}
\]
for $f\in\mathcal{M}$ and $\sigma\in W_0$, where $\widetilde{f}(z,\xi):=
f(\xi,z)$. Set $\widetilde{\nabla}=\nabla(\widetilde{D},\widetilde{\kappa},q)$.
Then 
\[
j\circ\nabla(w,\widetilde{w})=
\widetilde{\nabla}(\widetilde{w},w)\circ j\qquad
\forall\, (w,\widetilde{w})\in W\times\widetilde{W}.
\]
In particular, $j$ restricts to a complex linear isomorphism
\[
\textup{Sol}_{KZ}(D,\kappa,q)\overset{\sim}{\longrightarrow}
\textup{Sol}_{KZ}(\widetilde{D},
\widetilde{\kappa},q).
\]
\end{rema}

\subsection{Asymptotically free solutions}
\label{SOLsection}
For the $\textup{GL}_{n+1}$ root system datum, asymptotically free solutions of
quantum KZ equations have been constructed using correlation functions
for quantum affine algebras in \cite{FR}, see also \cite[\S 10]{EFK}.
In the present context we establish the existence of asymptotically free
solutions using classical asymptotic methods going back to Birkhoff
\cite{Bi} (see \cite[Appendix]{vMS} for a detailed discussion 
of this approach that fits the present context).

Repeating the arguments of \cite{vMS,vM,StSph} one obtains the following
asymptotically free solution of the bispectral quantum KZ equations.
For $\epsilon>0$ set
\[
B_\epsilon:=\{(z,\xi)\in E_{\mathbb{C}}\times E_{\mathbb{C}} \,\, | \,\,
|q^{-\alpha_i(z)}|, |q^{-\widetilde{\alpha}_i(\xi)}|<\epsilon
\quad \forall\, i\in\{1,\ldots,n\}\}.
\]
\begin{thm}\label{firststep}
There exists a unique  
$\Phi_{KZ}(\cdot,\cdot)=\Phi_{KZ}(\cdot,\cdot;D,\kappa,q)\in\mathcal{M}
\otimes_{\mathbb{C}}\mathcal{V}$ such that
\begin{enumerate}
\item $\Phi_{KZ}\in\textup{Sol}_{KZ}$,
\item for some $\epsilon>0$,
\[
\Phi_{KZ}(z,\xi)=\mathcal{W}(z,\xi)\sum_{(\alpha,\beta)\in
Q^+\times\widetilde{Q}^+}
\Upsilon_{(\alpha,\beta)}q^{-\alpha(z)-\beta(\xi)}
\]
for $(z,\xi)\in B_\epsilon$, with the 
$\mathcal{V}$-valued sum $\sum_{(\alpha,\beta)\in Q^+\times\widetilde{Q}^+}
\Upsilon_{(\alpha,\beta)}q^{-\alpha(z)-\beta(\xi)}$ 
($\Upsilon_{(\alpha,\beta)}\in\mathcal{V}$) converging normally
for $(z,\xi)$ in compacta of $B_\epsilon$, 
\item $\Upsilon_{(0,0)}=v_{w_0}$.
\end{enumerate}
\end{thm}
\begin{proof}
Compared to the proofs in \cite[Thm. 5.3]{vMS} and \cite[Thm. 5.4]{vM} 
an extra argument is needed to take care of the extra flexibility in
the choice of lattices $\Lambda$ and $\widetilde{\Lambda}$.

Since $Q\subseteq\Lambda$ and $\widetilde{Q}\subseteq\widetilde{\Lambda}$
there exist sublattices
\[
M:=\bigoplus_{i=1}^n\mathbb{Z}\varpi_i\subseteq\Lambda,\qquad
\widetilde{M}:=\bigoplus_{i=1}^n\mathbb{Z}\widetilde{\varpi}_i\subseteq
\widetilde{\Lambda}
\]
with $\varpi_i$ and $\widetilde{\varpi}_i$ satisfying
$\bigl(\varpi_i,\alpha_j^\vee\bigr)\in\delta_{i,j}\mathbb{Z}_{>0}$
and $\bigl(\widetilde{\varpi}_i,\widetilde{\alpha}_j^\vee\bigr)\in
\delta_{i,j}\mathbb{Z}_{>0}$ for $i,j\in\{1,\ldots,n\}$.
The arguments in \cite[Thm. 5.3]{vMS} and \cite[Thm. 5.4]{vM} 
now lead to the proof of the 
existence and uniqueness of a meromorphic $\mathcal{V}$-valued function 
$\Phi_{KZ}(\cdot,\cdot)$ satisfying (2), (3) and satisfying the compatible
system
\begin{equation}\label{BqKZres}
\nabla(\tau(\lambda),\tau(\nu))\Phi_{KZ}=\Phi_{KZ}
\qquad \forall\, (\lambda,\nu)\in M\times\widetilde{M}
\end{equation}
of diference equations.
Fix $(\lambda^\prime,\nu^\prime)\in\Lambda\times\widetilde{\Lambda}$
and set $\Phi_{KZ}^\prime:=
\nabla(\tau(\lambda^\prime),\tau(\nu^\prime))\Phi_{KZ}$.
By the integrability of the bispectral quantum KZ equations it follows that
$\Phi_{KZ}^\prime$ 
satisfies \eqref{BqKZres}. Since $\Phi_{KZ}^\prime$ also satisfies properties
(2) and (3) we conclude that $\Phi_{KZ}^\prime=\Phi_{KZ}$. 
Hence $\Phi_{KZ}\in\textup{Sol}_{KZ}$.
\end{proof}

It is now possible to establish various properties of $\Phi_{KZ}$
(duality, singularities) by a detailed analysis of the bispectral
quantum KZ equations. It leads to the following result.
\begin{prop}\label{propertiesKZ}
{\bf (i)} $\Phi_{KZ}$ is selfdual,
\[
\Phi_{KZ}(z,\xi;D,\kappa,q)=\Phi_{KZ}(\xi,z;\widetilde{D},\widetilde{\kappa},q).
\]
{\bf (ii)} The $\mathcal{V}$-valued meromorphic function
\[
\Psi_{KZ}(z,\xi):=\frac{\mathcal{S}(z)\widetilde{\mathcal{S}}(\xi)}
{\mathcal{W}(z,\xi)}\Phi_{KZ}(z,\xi)
\]
has a $\mathcal{V}$-valued series expansion
\[
\Psi_{KZ}(z,\xi)=\sum_{(\alpha,\beta)\in Q^+\times\widetilde{Q}^+}
\Gamma_{(\alpha,\beta)}^{KZ}q^{-\alpha(z)-\beta(\xi)},
\]
normally convergent for $(z,\xi)$ in compacta of $E_{\mathbb{C}}\times
E_{\mathbb{C}}$. In particular, $\Psi_{KZ}(z,\xi)$ is holomorphic
in $(z,\xi)\in E_{\mathbb{C}}\times E_{\mathbb{C}}$ and 
$\Gamma^{KZ}_{(0,0)}=\Upsilon_{(0,0)}=v_{w_0}$.\\
{\bf (iii)} Define for $\alpha\in Q^+$ the $\mathcal{V}$-valued
holomorphic function 
$\Gamma_\alpha^{KZ}(\xi)$ in $\xi\in E_{\mathbb{C}}$ by
\[
\Gamma_\alpha^{KZ}(\xi):=\sum_{\beta\in\widetilde{Q}^+}
\Gamma_{(\alpha,\beta)}^{KZ}q^{-\beta(\xi)},
\]
so that $\Psi_{KZ}(z,\xi)=\sum_{\alpha\in Q^+}
\Gamma_\alpha^{KZ}(\xi)q^{-\alpha(z)}$.
Then
\[
\Gamma_0^{KZ}(\xi)=\prod_{\alpha\in R_0^+}\bigl(q_\alpha^2q^{-2\widetilde{\alpha}(\xi)};
q_\alpha^2\bigr)_{\infty}v_{w_0}.
\]
{\bf (iv)} The bispectral quantum KZ equations are consistent,
\[
\textup{dim}_{\mathcal{F}}\bigl(\textup{Sol}_{KZ}\bigr)=
\textup{dim}_{\mathbb{C}}\bigl(\mathcal{V})=\#W_0.
\] 
Furthermore, $\{\nabla(e,\sigma)\Phi_{KZ}\}_{\sigma\in W_0}$ is 
a $\mathcal{F}$-basis
of $\textup{Sol}_{KZ}$.
\end{prop}
\begin{proof}
The proofs for {\bf (i)}, {\bf (ii)} and {\bf (iv)} 
as given in \cite{vMS,vM,StSph} for the twisted case
$\bullet=t$ with $\widetilde{\Lambda}=\Lambda$ generalize easily to the present
context (for {\bf (i)} use Remark \ref{DS}).\\
{\bf (iii)} Similarly as in \cite{vMS,vM} for the twisted equal lattice case, 
the asymptotics as $q^{-\alpha_i(z)}\rightarrow 0$ for 
$i=1,\ldots,n$ shows
that $\Gamma_0^{KZ}(\xi)=\widetilde{\mathcal{S}}(\xi)K(\xi)v_{w_0}$ 
for a unique scalar valued meromorphic function
$K(\xi)$ in $\xi\in E_{\mathbb{C}}$ having a convergent
power series expansion
\[
K(\xi)=\sum_{\beta\in\widetilde{Q}^+}k_\beta q^{-\beta(\xi)},\qquad
k_0=1
\]
if $|q^{-\widetilde{\alpha}_i(\xi)}|$ is sufficiently
small for all $i\in\{1,\ldots,n\}$.

In the twisted equal lattice case \cite{vMS,vM,StSph}, $K(\xi)$ was explicitly
determined using the difference Che\-red\-nik-Matsuo correspondence, which
puts the problem in the context of the bispectral problem of the
(higher order) RMKC operators. 
We give here a new proof, 
which stays completely in the realm of the 
bispectral quantum KZ equations.

We characterize $K(\xi)$ as a formal power series
in $q^{-\widetilde{\alpha}_1(\xi)},\ldots,q^{-\widetilde{\alpha}_n(\xi)}$
with constant coefficient $1$ and solving 
an explicit system of difference equations in $\xi$. To derive the
difference equations, consider for $\lambda\in\Lambda^+$ the meromorphic
function
\[
R_\lambda(z,\xi):=q^{(\widetilde{\rho}+w_0z,\lambda)}
\bigl(C_{(e,\tau(\lambda))}(z,\xi)v_{w_0}\bigr)|_{v_{w_0}},
\]
where $v|_{v_{w_0}}$ for $v\in\mathcal{V}$ means picking the $v_{w_0}$-component
in the expansion of $v$ as linear combination of the basis elements
$v_\sigma$ ($\sigma\in W_0$) of $\mathcal{V}$. 
Then
\[
R_\lambda(z,\xi)=\sum_{\alpha\in Q^+}q^{-\alpha(z)}r_\lambda^{(\alpha)}(\xi)
\]
(finite sum) 
with $r_\lambda^{(\alpha)}\in\mathbb{C}[[q^{-\widetilde{\alpha}_1},
\ldots,q^{-\widetilde{\alpha}_n}]]$,
which follows from the extension of \cite[Lem. 5.3]{vM} to the present
setup. As limit of the bispectral quantum KZ equations for
$\Phi_{KZ}(z,\xi)$ it follows that 
$K(\xi)$ satisfies the difference equations
\[
r_\lambda^{(0)}(\xi)K(\xi-\lambda)=K(\xi)\qquad \forall\,\lambda\in\Lambda^+,
\]
which characterize $K(\xi)$ as formal power series
in the $q^{-\widetilde{\alpha}_i(\xi)}$ with constant coefficient $1$. 

Choosing a reduced expression of $\tau(\lambda)\in
\widetilde{W}$ and using the cocycle condition \eqref{cocC}
allows one to give an
explicit expression of $C_{(e,\tau(\lambda))}(z,\xi)$, from which it follows that
\[
r_\lambda^{(0)}(\xi)=\prod_{a\in\widetilde{R}^{\bullet,+}\cap\tau(\lambda)
\widetilde{R}^{\bullet,-}}c_a(\xi;-\widetilde{\kappa},q)^{-1}
\]
for $\lambda\in\Lambda^+$. Consequently
\begin{equation*}
\begin{split}
K(\xi)&=\prod_{\alpha\in R_0^-, r\in\mathbb{Z}_{>0}}c_{\widetilde{\alpha}^{(r)}}(\xi;
-\widetilde{\kappa},q)^{-1}\\
&=\widetilde{\mathcal{S}}(\xi)^{-1}\prod_{\alpha\in R_0^+}
\bigl(q_\alpha^2q^{-2\widetilde{\alpha}(\xi)};q_\alpha^2\bigr)_{\infty}.
\end{split}
\end{equation*}
\end{proof}
\subsection{Ruijsenaars-Macdonald-Koornwinder-Cherednik 
operators}\label{RMKC}
We follow Cherednik's \cite{Cann} construction of higher order 
Ruijsenaars-Macdonald-Koornwinder-Cherednik (RMKC) operators, 
see also \cite{M, StB}. For the precise definition of the affine Hecke
algebra in the present context, we refer to \cite[\S 2.4]{StB}.

Let $\widehat{w}$ be the contragredient action of $w\in W$ on
meromorphic functions on $E_{\mathbb{C}}$,
\[
(\widehat{w}f)(z):=f(w^{-1}z).
\]
For $i\in\{0,\ldots,n\}$ the Demazure-Lusztig type difference reflection
operators
\begin{equation}\label{DL}
\widehat{T}_i:=q^{\kappa_i}+q^{-\kappa_i}c_i(\cdot;\kappa,q)
(\widehat{s}_i-\textup{id})
\end{equation}
define a representation of the affine Hecke algebra $H(W^\bullet;q^\kappa)$
on the space of meromorphic functions on $E_{\mathbb{C}}$,
where $q^\kappa$ stands for the Hecke parameters $q^{\kappa_i}$ ($i=0,\ldots,n)$.
The corresponding Hecke relation is
\[
(\widehat{T}_i-q^{\kappa_i})(\widehat{T}_i+q^{-\kappa_i})=0.
\]
Recall that $W\simeq\Omega\ltimes W^\bullet$ with $\Omega\subset W$ the subgroup
consisting of extended affine Weyl group elements of length zero.
Then \eqref{DL} and the operators $\widehat{u}$ ($u\in\Omega$) provide a 
representation of the extended affine Hecke algebra $H(W;q^\kappa)\simeq
\Omega\ltimes H(W^\bullet;q^\kappa)$, cf. \cite[\S 2.4]{StB}.

Fix $\nu\in\widetilde{\Lambda}^+$ and suppose that 
$\tau(\nu)=s_{i_1}\cdots s_{i_r}u\in W$ is a reduced expression
($0\leq i_j\leq n$, $u\in\Omega$). The associated operator
\[
\widehat{Y}^\nu:=\widehat{T}_{i_1}\cdots\widehat{T}_{i_r}\widehat{u}
\]
is well defined and invertible. 
For arbitrary weight $\nu\in\widetilde{\Lambda}$ the 
Bernstein-Zelevinsky operator is defined as 
\[
\widehat{Y}^\nu:=\widehat{Y}^{\nu_1}\bigl(\widehat{Y}^{\nu_2}\bigr)^{-1},
\] 
where the $\nu_i\in\widetilde{\Lambda}^+$ are such that
$\nu=\nu_1-\nu_2$. The operators $\widehat{Y}^\nu$
($\nu\in\widetilde{\Lambda}$) are well defined and mutually commute.

Fix $\nu\in\widetilde{\Lambda}^+$. There exists unique difference
operators $L_{\nu,\sigma}$ ($\sigma\in W_0$)
such that 
\[
\sum_{\nu^\prime\in W_0\nu}\widehat{Y}^{\nu^\prime}=\sum_{\sigma\in W_0}L_{\nu,\sigma}
\widehat{\sigma}.
\]
\begin{defthm}[\cite{Cann}]
The difference operators 
\[
L_{\nu}:=\sum_{w\in W_0}L_{\nu,\sigma},\qquad \nu\in\widetilde{\Lambda}^+
\]
are the higher order RMKC operators associated to the initial datum
$(D,\kappa,q)$. They are $W_0$-equivariant and mutually commute.
\end{defthm}
The difference operators $L_{\nu}$ can be made entirely explicit for
miniscule dominant weights 
$\nu\in \widetilde{\Lambda}^+_{min}$ and for the quasi-miniscule dominant
weight $\nu=\widetilde{\psi}$, see, e.g., \cite{Cann,M,StB}. We give here
only the explicit formula for $\nu=\widetilde{\psi}$,
\begin{equation*}
\begin{split}
\bigl(L_{\widetilde{\psi}}f\bigr)(z)&=
q^{-(\rho,\widetilde{\psi})}\sum_{w\in W_0/W_{0,\psi}}
c_{\tau(\widetilde{\psi})}(w^{-1}z;\kappa,q)\bigl(f(z+w\widetilde{\psi})-
f(z)\bigr)\\
&\qquad\qquad
+\Bigl(\sum_{w\in W_0/W_{0,\psi}}q^{-(\rho,w\widetilde{\psi})}\Bigr)f(z),
\end{split}
\end{equation*}
where for $w\in W$,
\[
c_w(z;\kappa,q):=\prod_{a\in R^{\bullet,+}\cap w^{-1}R^{\bullet,-}}
c_a(z;\kappa,q).
\]
Recall the explicit difference operator $L$ from Subsection \ref{IDE}.
\begin{lem}
$L_{\widetilde{\psi}}=L$.
\end{lem}
\begin{proof}
It suffices to show that $c_{\tau(\widetilde{\psi})}(z;\kappa,q)=A(z)$,
with $A(z)$ the trigonometric function \eqref{A(z)}. This follows from
the fact that
\[
R^{\bullet,+}\cap \tau(-\widetilde{\psi})R^{\bullet,-}=
\{\psi,\psi^{(1)}\}\cup\{\alpha\in R_0^+ \,\, | \,\, 
\bigl(\widetilde{\psi},\widetilde{\alpha}^\vee\bigr)=1
\}.
\]
\end{proof}

\subsection{The bispectral problem for RMKC operators}\label{biRMKC}
We denote the higher order RMKC operators with respect to the dual initial datum
$(\widetilde{D},\widetilde{\kappa},q)$ by
$\widetilde{L}_{\lambda}$ ($\lambda\in\Lambda^+$). Note that the 
RMKC operator $\widetilde{L}:=\widetilde{L}_\theta$ is explicitly given by
\[
\widetilde{L}=q^{-(\widetilde{\rho},\theta)}\sum_{w\in W_0/W_{0,\theta}}
\widetilde{A}(w^{-1}z)\bigl(f(z+w\theta)-f(z)\bigr)+
\Bigl(\sum_{w\in W_0/W_{0,\theta}}q^{-(\widetilde{\rho},w\theta)}\Bigr)f(z),
\]
with
\begin{equation*}
\begin{split}
\widetilde{A}(z)&=\frac{(1-\widetilde{a}_\theta q^{\widetilde{\theta}(z)})
(1-\widetilde{b}_\theta q^{\widetilde{\theta}(z)})
(1-\widetilde{c}_\theta q^{\widetilde{\theta}(z)})
(1-\widetilde{d}_\theta q^{\widetilde{\theta}(z)})}
{(1-q^{2\widetilde{\theta}(z)})(1-q_\theta^2q^{2\widetilde{\theta}(z)})}\\
&\times
\prod_{\alpha\in R_0^+: (\theta,\alpha^\vee)=1}
\frac{(1-\widetilde{a}_\alpha q^{\widetilde{\alpha}(z)})
(1-\widetilde{b}_\alpha q^{\widetilde{\alpha}(z)})}
{(1-q^{2\widetilde{\alpha}(z)})}
\end{split}
\end{equation*}
(here we use that $\widetilde{a}_\alpha=a_{\widetilde{\alpha}}(\widetilde{D},
\widetilde{\kappa},q),\ldots,\widetilde{d}_\alpha=d_{\widetilde{\alpha}}(
\widetilde{D},\widetilde{\kappa},q)$).

\begin{defi}
The system
\begin{equation}\label{BS}
\begin{split}
L_\nu f(\cdot,\xi)&=\Bigl(\sum_{w\in W_0/W_{0,\nu}}q^{(w\nu,\xi)}\Bigr)f(\cdot,\xi),
\qquad \nu\in\widetilde{\Lambda}^+,\\
\widetilde{L}_\lambda f(z,\cdot)&=\Bigl(\sum_{w\in W_0/W_{0,\lambda}}
q^{(w\lambda,z)}\Bigr)f(z,\cdot),\qquad \lambda\in\Lambda^+
\end{split}
\end{equation}
of difference equations for $f\in\mathcal{M}$ is called the
bispectral problem of the RMKC operators.
We write $\textup{Sol}_{RMKC}$ for the vector space over $\mathcal{F}$
consisting of $f\in\mathcal{M}$ satisfying \eqref{BS}.
\end{defi}
Since the (higher order) RMKC operators are $W_0$-equivariant,
$\textup{Sol}_{RMKC}\subset\mathcal{M}$ is $W_0\times W_0$-invariant
with respect to the contragredient action of $W_0\times W_0$ on $\mathcal{M}$.

\subsection{The difference Cherednik-Matsuo correspondence}\label{CMK}
Recall that $\mathcal{M}\otimes_{\mathbb{C}}\mathcal{V}$ is a 
left $W_0\times W_0$-module 
with respect to the $\nabla$-action \eqref{nablaaction}.
Let $\chi: \mathcal{M}\otimes_{\mathbb{C}}
\mathcal{V}\rightarrow\mathcal{M}$ be the $\mathcal{M}$-linear $W_0\times
W_0$-equivariant map
\[
\chi\bigl(\sum_{w\in W_0}f_\sigma\otimes v_\sigma\bigr):=
q^{\kappa_{w_0}}\sum_{w\in W_0}q^{-\kappa_w}f_w,
\]
where $\kappa_w:=\sum_{\alpha\in R_0^+\cap w^{-1}R_0^-}\kappa_\alpha$ for $w\in W_0$.
\begin{thm}
The map $\chi$ restricts to an injective, $\mathcal{F}$-linear,
$W_0\times W_0$-equivariant map
\[
\chi: \textup{Sol}_{KZ}\hookrightarrow \textup{Sol}_{RMKC}.
\]
\end{thm}
\begin{proof}
The proofs in \cite{vMS,vM}, extending the results of Cherednik \cite{CQKZ2}
to the bispectral setting, generalize easily to the present setting.
In particular, 
the injectivity is proved by showing that the meromorphic functions
$\chi(\nabla(e,w)\Phi_{KZ}\bigr)$ ($w\in W_0$) are $\mathcal{F}$-linear
independent.
\end{proof}

\subsection{Basic Harish-Chandra series}\label{HCsub}

\begin{defi}
The basic Harish-Chandra series $\Phi(\cdot,\cdot)=\Phi(\cdot,\cdot;D,\kappa,q)$
is defined by
\[
\Phi:=\chi\bigl(\Phi_{KZ}\bigr)\in\textup{Sol}_{RMKC}.
\]
\end{defi}
We are now in the position to prove all the fundamental properties of
the basic Harish-Chandra series as stated in Section \ref{Intro}.\\

\noindent
{\bf Proof of Theorem \ref{THM1}.}
The results on the asymptotically free
solution $\Phi_{KZ}$ of the bispectral quantum KZ equations 
from Subsection \ref{SOLsection} show that
the basic Harish-Chandra series $\Phi$
satisfies all the properties as stated in Theorem \ref{THM1}, with
$\Gamma_\alpha(\xi)=\chi\bigl(\Gamma_\alpha^{KZ}(\xi)\bigr)$ 
($\alpha\in Q^+$). In particular, the eigenvalue equation \eqref{EE}
for $\Phi$ follows from the fact that $\Phi$ solves the bispectral problem
\eqref{BS} of the RMKC operators since $L=L_{\widetilde{\psi}}$.
It thus suffices to prove the uniqueness claim. 

This follows from the results in \cite[\S 3]{CInd} (untwisted case) and 
\cite[\S 2]{LS} on the analog of the Harish-Chandra homomorphism and from
the subsequent formal analysis of the basic Harish-Chandra series in
\cite[\S 4]{LS}. These results show
that the eigenvalue equations \eqref{EE} for $\Phi$ is equivalent
to a system of recurrence relations for the expansion coefficients
$\Gamma_\alpha(\xi)$ ($\alpha\in Q^+$) of the form
\begin{equation*}
\Bigl(\sum_{w\in W_0/W_{0,\psi}}q^{(w\widetilde{\psi},\xi)}-
\sum_{w\in W_0/W_{0,\psi}}q^{(w\widetilde{\psi},\xi+w_0\alpha)}\Bigr)
\Gamma_\alpha(\xi)=
\sum_{\stackrel{\beta\in Q^+\setminus\{0\}:}{\alpha-\beta\in Q^+}}d_{\alpha,\beta}(\xi)
\Gamma_{\alpha-\beta}(\xi)
\end{equation*}
for suitable holomorphic functions $d_{\alpha,\beta}(\xi)$ in $\xi\in 
E_{\mathbb{C}}$. Hence the eigenvalue equations \eqref{EE} determine the
expansion coefficients $\Gamma_\alpha$ ($\alpha\in 
Q^+\setminus\{0\}$) uniquely in
terms of $\Gamma_0$. 
$\qquad\qquad\qquad\qquad\qquad\qquad\qquad\qquad\qquad\qquad\,\,
\qquad\qquad\qquad\qquad\qquad\qquad\,\,\Box$

\noindent
{\bf Proof of Theorem \ref{THM2}.} The selfduality of the basic Harish-Chandra
series $\Phi$ follows immediately from the selfduality of $\Phi_{KZ}$
(see Proposition \ref{propertiesKZ}{\bf (i)}).
$\qquad\qquad\quad\,\Box$

\noindent
{\bf Proof of Theorem \ref{THM3}.} By Proposition 
\ref{propertiesKZ}{\bf (iv)} there exists unique
$m_{\tau_1,\tau_2}^{\sigma}\in\mathcal{F}$ ($\sigma,\tau_1,\tau_2\in W_0$)
such that
\[
\nabla(\sigma,\tau_2)\Phi_{KZ}=
\sum_{\tau_1\in W_0}m_{\tau_1,\tau_2}^{\sigma}\bigl(\nabla(e,\tau_1)\Phi_{KZ}
\bigr)
\]
for all $\sigma,\tau_2\in W_0$. Applying the injective, 
$W_0\times W_0$-equivariant Cherednik map $\chi$ shows
that
\begin{equation}\label{cc}
\Phi(\sigma^{-1}z,\tau_2^{-1}\xi)=\sum_{\tau_1\in W_0}
m_{\tau_1,\tau_2}^{\sigma}(z,\xi)\Phi(z,\tau_1^{-1}\xi)
\end{equation}
as meromorphic functions in
$(z,\xi)\in E_{\mathbb{C}}\times E_{\mathbb{C}}$.
By the injectivity of $\chi|_{\textup{Sol}_{KZ}}$ it follows 
that the equations \eqref{cc} determine 
the $m_{\tau_1,\tau_2}^{\sigma}\in\mathcal{F}$
uniquely. 
$\qquad\qquad\qquad\qquad\quad\,\,\Box$

\noindent
{\bf Proof of Proposition \ref{PROP5}.}
Theorem \ref{THM1} implies that
\begin{equation}\label{EC}
\Phi(z,\xi)=q^{-(\rho+w_0\xi,z)}\sum_{\alpha\in Q^+}\widehat{\Gamma}_\alpha(\xi)
q^{-(\alpha,z)}
\end{equation}
for generic $\xi\in E_{\mathbb{C}}$ if $\Re(z)$ is
sufficiently deep in the negative fundamental Weyl chamber $E_-$, 
with leading coefficient
\[
\widehat{\Gamma}_0(\xi)=q^{(\widetilde{\rho},\rho-\xi)}
\frac{\Gamma_0(\xi)}{\widetilde{\mathcal{S}}(\xi)}.
\]
Such (formal) power series solutions of the 
spectral problem of the RMKC operators are unique up to normalization
(see \cite[Thm. 4.4]{LS}). These two observations are valid without
any restrictions on the initial datum $(D,\kappa,q)$. 
Under the assumptions on $(D,\kappa,q)$ as stated in the proposition, 
comparison with the series expansion of the multivariable Baker-Akhiezer
function $\psi(-w_0\xi,\cdot)$ from
\cite[Rem. 3.6]{ChE} shows that
$\Phi(z,\xi)=\textup{cst}(\xi)\psi(-w_0\xi,z)$.
A straightforward computation proves 
that the leading coefficient $\widehat{\Gamma}_0(\xi)$ of the power
series expansion \eqref{EC} of $\Phi(\cdot,\xi)$ 
coincides with
the leading coefficient $\Delta^\prime(\xi)=\Delta^\prime(-w_0\xi)$ 
of $\psi(-w_0\xi,\cdot)$ up to a nonzero multiplicative constant
(see \cite[\S 2.1.4]{ChE} for the definition of $\Delta^\prime(\xi$)). 
This shows that $\textup{cst}(\xi)$ is independent of $\xi$.
$\qquad\qquad\qquad\qquad\qquad\qquad\qquad\qquad
\qquad\qquad\qquad\qquad\qquad\qquad\quad\Box$

\begin{rema}
In the proof of Proposition \ref{PROP5} we could have used the selfduality
of $\Phi$ (Theorem \ref{THM2}) and $\psi$ (\cite[Thm. 3.3(iii)]{ChE})
to immediately conclude that $\textup{cst}(\xi)$ is independent of $\xi$.
The current proof has the advantage that the selfduality of the normalized
multivariable Baker-Akhiezer function $\psi$ becomes a 
consequence of the selfduality of $\Phi$. 
\end{rema}

\section{The  connection cocycle}\label{CC}
In this section we prove the explicit expressions for the
connection coefficients as stated in
Theorem \ref{THM4} using rank reduction. The strategy is as follows.

Fix $i\in\{1,\ldots,n\}$. Let $i^*\in\{1,\ldots,n\}$ be 
the corresponding index such that
$-w_0\alpha_i=\alpha_{i^*}$. Let $\widetilde{\delta}_i\in\widetilde{\Lambda}$
be a weight such that $(\widetilde{\delta}_i,\widetilde{\alpha}_i^\vee)=0$
and $(\widetilde{\delta}_i,\widetilde{\alpha}_j^\vee\bigr)>0$ if 
$j\not=i$ (one can take for instance $\widetilde{\delta}_i=
\sum_{j\not=i}\widetilde{\varpi}_j$ with $\widetilde{\varpi}_j\in
\widetilde{\Lambda}$ as in the proof of Theorem \ref{firststep}).
 
Recall the holomorphic function
\[
\Psi(z,\xi)=\sum_{\alpha\in Q^+}\Gamma_\alpha(\xi)q^{-\alpha(z)}
\]
from Theorem \ref{THM1}, such that
\begin{equation}\label{altexp}
\Phi(z,\xi)=\frac{\mathcal{W}(z,\xi)}
{\mathcal{S}(z)\widetilde{\mathcal{S}}(\xi)}\Psi(z,\xi).
\end{equation}
Define the holomorphic function $\mathcal{S}_i(x)$ in $x\in\mathbb{C}$ by
\[
\mathcal{S}_i(x):=\bigl(q_i^2a_i^{-1}q^{-x},q_i^2b_i^{-1}q^{-x},
q_i^2c_i^{-1}q^{-x},q_i^2d_i^{-1}q^{-x};q_i^2\bigr)_{\infty}
\]
and the holomorphic function $\Psi_i(x,\xi)$ in $(x,\xi)\in\mathbb{C}\times 
E_{\mathbb{C}}$ by
\[
\Psi_i(x,\xi):=\sum_{r=0}^{\infty}\Gamma_{r\alpha_i}(\xi)q^{-rx}.
\]
Then
\begin{equation*}
\begin{split}
\lim_{m\rightarrow\infty}\mathcal{S}(z-m\widetilde{\delta}_i)&=
\mathcal{S}_i(\alpha_i(z)),\\
\lim_{m\rightarrow\infty}\Psi(z-m\widetilde{\delta}_i,\xi)&=
\Psi_i(\alpha_i(z),\xi),
\end{split}
\end{equation*}
uniformly on compacta. We define now the meromorphic function
$\Phi_i(x,\xi)$ in $(x,\xi)\in\mathbb{C}\times E_{\mathbb{C}}$ by
\[
\Phi_i(x,\xi):=\frac{\mathcal{W}_i(x,\widetilde{\alpha}_{i^*}(\xi))}
{\mathcal{S}_i(x)\widetilde{\mathcal{S}}(\xi)}\Psi_i(x,\xi)
\]
with the one variable plane wave
\[
\mathcal{W}_i(x,y):=q^{\frac{1}{2\mu_i}(\kappa_i+\kappa_{2\alpha_i}-x)
(\kappa_i+\kappa_{\alpha_i^{(1)}}-y)},
\]
cf. \cite[(2.1)]{StQ}. 
We will prove that
$\Phi_i(\cdot,\xi)$ is the asymptotically free solution of the 
Askey-Wilson \cite{AW} second
order difference operator, with associated AW parameters given by
$(a_i,b_i,c_i,d_i)$. This allows us to compute
the connection coefficients using results from the classical theory
\cite{GR} on basic
hypergeometric series.

\subsection{Vanishing connection coefficients}
We first show that most of the connection coefficients are zero.

\begin{prop}\label{zero}
Let $\tau_1,\tau_2\in W_0$.
Then $m_{\tau_1,\tau_2}^{s_i}\equiv 0$ if $\tau_1\not\in\{\tau_2,\tau_2s_{i^*}\}$.
\end{prop}
\begin{proof}
Fix $\sigma\not\in\{e,s_{i^*}\}$.
Since $m_{\tau_1,\tau_2}^{\sigma}(z,\xi)=
m_{\tau_2^{-1}\tau_1,e}^{\sigma}(z,\tau_2^{-1}\xi)$ it suffices to show
that $m_{\sigma,e}^{s_i}\equiv 0$.

Rewriting the identity
\[
\Phi(s_iz,\xi)=\sum_{\tau\in W_0}m_{\tau,e}^{s_i}(z,\xi)\Phi(z,\tau^{-1}\xi)
\]
using \eqref{altexp}
gives
\begin{equation}\label{a1}
\Psi(s_iz,\xi)=\frac{\mathcal{S}(s_iz)\widetilde{\mathcal{S}}(\xi)}
{\mathcal{S}(z)}\sum_{\tau\in W_0}\frac{
m_{\tau,e}^{s_i}(z,\xi)\mathcal{W}(z,\tau^{-1}\xi)}{\widetilde{\mathcal{S}}
(\tau^{-1}\xi)\mathcal{W}(s_iz,\xi)}\Psi(z,\tau^{-1}\xi).
\end{equation}
In \eqref{a1} we replace $z$ by $z-m\widetilde{\delta}_i$
and multiply the resulting identity by 
\[
q^{m(w_0\widetilde{\delta}_i,
\xi-\sigma^{-1}\xi)}\frac{\mathcal{S}(z-m\widetilde{\delta}_i)}
{\mathcal{S}(s_iz-m\widetilde{\delta}_i)}. 
\]
It gives
\begin{equation}\label{a2}
\begin{split}
&\frac{\mathcal{S}(z-m\widetilde{\delta}_i)}
{\mathcal{S}(s_iz-m\widetilde{\delta}_i)}
q^{m(w_0\widetilde{\delta}_i,\xi-\sigma^{-1}\xi)}\Psi(s_iz-m\widetilde{\delta}_i,
\xi)=\\
&=\frac{\widetilde{\mathcal{S}}(\xi)}{\mathcal{W}(s_iz,\xi)}
\sum_{\tau\in W_0}\frac{m_{\tau,e}^{s_i}(z,\xi)
\mathcal{W}(z,\tau^{-1}\xi)}{\widetilde{\mathcal{S}}(\tau^{-1}\xi)}
q^{m(w_0\widetilde{\delta}_i,\tau^{-1}\xi-\sigma^{-1}\xi)}
\Psi(z-m\widetilde{\delta}_i,\tau^{-1}\xi)
\end{split}
\end{equation}
since
$m_{\tau,e}^{s_i}$ is $\widetilde{\Lambda}\times\Lambda$-translation
invariant and 
\[
\frac{\mathcal{W}(z-m\widetilde{\delta}_i,\tau^{-1}\xi)}
{\mathcal{W}(s_iz-m\widetilde{\delta}_i,\xi)}=
q^{m(w_0\widetilde{\delta}_i,\tau^{-1}\xi-\xi)}
\frac{\mathcal{W}(z,\tau^{-1}\xi)}{\mathcal{W}(s_iz,\xi)}.
\]
Set
\[
E_{\mathbb{C}}^+:=\{\xi\in E_{\mathbb{C}}\,\, | \,\, \Re\bigl(\alpha_i(\xi)\bigr)>0
\quad \forall\, i \}
\]
and fix generic $(z,\xi)\in E_{\mathbb{C}}\times\sigma E_{\mathbb{C}}^+$.
Taking the limit $m\rightarrow\infty$ in \eqref{a2} then gives
\begin{equation}\label{a3}
\sum_{\tau\in\{\sigma,\sigma s_{i^*}\}}
\frac{m_{\tau,e}^{s_i}(z,\xi)\mathcal{W}(z,\tau^{-1}\xi)}
{\widetilde{\mathcal{S}}(\tau^{-1}\xi)}\Psi_i(\alpha_i(z),\tau^{-1}\xi)=0.
\end{equation}
Let $\widetilde{\varpi}_i\in\widetilde{\Lambda}$ such that
$(\widetilde{\varpi}_i,\widetilde{\alpha}_j^\vee)\in\delta_{i,j}\mathbb{Z}_{>0}$.
Replace $z$ by
$z-m\widetilde{\varpi}_i$ in \eqref{a3} and multiply both sides of the identity
by $q^{-m(\rho+w_0\sigma^{-1}\xi,\widetilde{\varpi}_i)}$. Then
\[
\sum_{\tau\in\{\sigma,\sigma s_{i^*}\}}
\frac{m_{\tau,e}^{s_i}(z,\xi)\mathcal{W}(z,\tau^{-1}\xi)}
{\widetilde{\mathcal{S}}(\tau^{-1}\xi)}
q^{m(w_0(\tau^{-1}\xi-\sigma^{-1}\xi),\widetilde{\varpi}_i)}
\Psi_i(\alpha_i(z)-m(\widetilde{\varpi}_i,\alpha_i),\tau^{-1}\xi)=0.
\]
Taking the limit $m\rightarrow\infty$ we get
\[
\frac{m_{\sigma,e}^{s_i}(z,\xi)\mathcal{W}(z,\sigma^{-1}\xi)
\Gamma_0(\sigma^{-1}\xi)}
{\widetilde{\mathcal{S}}(\sigma^{-1}\xi)}=0.
\]
Hence $m_{\sigma,e}^{s_i}(z,\xi)=0$ for generic 
$(z,\xi)\in E_{\mathbb{C}}\times \sigma E_{\mathbb{C}}^+$.
Since $m_{\sigma,e}^{s_i}$ is $\widetilde{\Lambda}\times\Lambda$-translation
invariant, we get $m_{\sigma,e}^{s_i}\equiv 0$. 
\end{proof}

\subsection{Rank reduction}
The aim is to compute $\Phi_i(x,\xi)$ explicitly in terms of
basic hypergeometric series. The following proposition is fundamental.

\begin{prop}\label{XYX}
{\bf (i)} We have
\[
\Phi_i(-\alpha_i(z),\xi)=m_{e,e}^{s_i}(z,\xi)\Phi_i(\alpha_i(z),\xi)+
m_{s_{i^*},e}^{s_i}(z,\xi)\Phi_i(\alpha_i(z),s_{i^*}\xi)
\]
as meromorphic functions in $(z,\xi)\in E_{\mathbb{C}}\times E_{\mathbb{C}}$.\\
{\bf (ii)} We have the eigenvalue equations
\begin{equation}\label{AWde}
\mathcal{M}_i\Phi_i(\cdot,\xi)=\bigl(q^{\widetilde{\alpha}_{i^*}(\xi)}+
q^{-\widetilde{\alpha}_{i^*}(\xi)}-\widetilde{a}_i-\widetilde{a}_i^{-1}\bigr)
\Phi_i(\cdot,\xi)
\end{equation}
as meromorphic functions in $(\cdot,\xi)\in\mathbb{C}\times E_{\mathbb{C}}$, 
with $\mathcal{M}_i$ the
Askey-Wilson \cite{AW} second order difference operator
\begin{equation*}
\begin{split}
\bigl(\mathcal{M}_ig\bigr)(x):=&A_i(x)(g(x-2\mu_i)-g(x))+
A_i(-x)(g(x+2\mu_i)-g(x)),\\
A_i(x):=&\frac{(1-a_iq^{-x})(1-b_iq^{-x})(1-c_iq^{-x})(1-d_iq^{-x})}
{\widetilde{a}_i(1-q^{-2x})(1-q_i^2q^{-2x})}.
\end{split}
\end{equation*}
\end{prop}
\begin{proof}
{\bf (i)} By Proposition \ref{zero},
\[
\Phi(s_iz,\xi)=m_{e,e}^{s_i}(z,\xi)\Phi(z,\xi)+
m_{s_{i^*},e}^{s_i}(z,\xi)\Phi(z,s_{i^*}\xi).
\]
The result now follows by multiplying both sides of this identity by
\[
\frac{\mathcal{W}_i(-\alpha_i(z),\widetilde{\alpha}_{i^*}(\xi))}
{\mathcal{W}(s_iz,\xi)},
\]
replacing $z$ by $z-m\widetilde{\delta}_i$ and taking the limit
$m\rightarrow\infty$, using that the connection coefficients
are $\widetilde{\Lambda}\times\Lambda$-translation invariant and that
\[
\frac{\mathcal{W}(s_iz,\xi)}{\mathcal{W}(z,\xi)}=
\frac{\mathcal{W}_i(-\alpha_i(z),\widetilde{\alpha}_{i^*}(\xi))}
{\mathcal{W}_i(\alpha_i(z),\widetilde{\alpha}_{i^*}(\xi))},\qquad
\frac{\mathcal{W}(s_iz,\xi)}{\mathcal{W}(z,s_{i^*}\xi)}=
\frac{\mathcal{W}_i(-\alpha_i(z),\widetilde{\alpha}_{i^*}(\xi))}
{\mathcal{W}_i(\alpha_i(z),-\widetilde{\alpha}_{i^*}(\xi))}.
\]
{\bf (ii)} Define a second difference operator $\mathcal{N}_i$ by
\begin{equation*}
\begin{split}
\bigl(\mathcal{N}_ig\bigr)(x)&:=B_i(x)g(x-\mu_i)+B_i(-x)g(y+\mu_i),\\
B_i(x)&:=\frac{(1-a_iq^{-x})(1-b_iq^{-x})}{q^{\kappa_i}(1-q^{-2x})}.
\end{split}
\end{equation*}
{}From (the proof of) \cite[Prop. 3.13]{LS} we obtain the following result.
\begin{enumerate}
\item[{\bf Case a.}] If $\alpha_i\not\in W_0\psi$ then
$\bigl(\widetilde{Q},\widetilde{\alpha}_i^\vee\bigr)=\mathbb{Z}$ and 
\begin{equation}\label{b1}
\mathcal{N}_i\Phi_i(\cdot,\xi)=\bigl(q^{\widetilde{\alpha}_{i^*}(\xi)/2}+
q^{-\widetilde{\alpha}_{i^*}(\xi)/2}\bigr)\Phi_i(\cdot,\xi)
\end{equation}
as meromorphic functions in $(\cdot,\xi)\in\mathbb{C}\times E_{\mathbb{C}}$.
\item[{\bf Case b.}] If $\alpha_i\in W_0\psi$ and 
$\bigl(\widetilde{Q},\widetilde{\alpha}_i^\vee\bigr)=\mathbb{Z}$ then
\begin{equation}\label{b2}
\bigl(\mathcal{N}_i+q^{\widetilde{\alpha}_{i^*}(\xi)/2}+
q^{-\widetilde{\alpha}_{i^*}(\xi)/2}+\textup{cst}_i\bigr)
\bigl(\mathcal{N}_i-q^{\widetilde{\alpha}_{i^*}(\xi)/2}
-q^{-\widetilde{\alpha}_{i^*}(\xi)}\bigr)\Phi_i(\cdot,\xi)=0
\end{equation}
as meromorphic functions in $\mathbb{C}\times E_{\mathbb{C}}$, 
where
\[
\textup{cst}_i:=q^{\widetilde{\alpha}_{i^*}(\xi)/2}
\sum_{\sigma}q^{-\widetilde{\psi}(\sigma^{-1}w_0\xi)}
\]
with the sum running over the $\sigma\in W_0/W_{0,\psi}$ satisfying
$\bigl(\sigma\widetilde{\psi},\widetilde{\alpha}_i^\vee\bigr)=-1$.
\item[{\bf Case c.}] If $\alpha_i\in W_0\psi$ and
$\bigl(\widetilde{Q},\widetilde{\alpha}_i^\vee\bigr)=2\mathbb{Z}$ then
\[
\mathcal{M}_i\Phi_i(\cdot,\xi)=\bigl(q^{\widetilde{\alpha}_{i^*}(\xi)}+
q^{-\widetilde{\alpha}_{i^*}(\xi)}-\widetilde{a}_i-\widetilde{a}_i^{-1}\bigr)
\Phi_i(\cdot,\xi)
\]
as meromorphic functions in $(\cdot,\xi)\in \mathbb{C}\times E_{\mathbb{C}}$.
\end{enumerate}
It thus suffices to show that in case {\bf a} (resp. case {\bf b}),
\eqref{b1} (resp. \eqref{b2}) implies \eqref{AWde}.

In both cases {\bf a} and {\bf b} we have
$\kappa_{\alpha_i^{(1)}}=\kappa_{\alpha_i}$,
$\kappa_{2\alpha_i^{(1)}}=\kappa_{2\alpha_i}$ 
since $\bigl(\widetilde{\Lambda},\widetilde{\alpha}_i^\vee\bigr)=\mathbb{Z}$,
hence $c_i=q_ia_i$, $d_i=q_ib_i$ for the corresponding AW parameters.
Consequently
\begin{equation*}
\begin{split}
\mathcal{M}_i-q^{\widetilde{\alpha}_{i^*}(\xi)}-
&q^{-\widetilde{\alpha}_{i^*}(\xi)}+\widetilde{a}_i+\widetilde{a}_i^{-1}=\\
&=\bigl(\mathcal{N}_i+q^{\widetilde{\alpha}_{i^*}(\xi)/2}+
q^{-\widetilde{\alpha}_{i^*}(\xi)/2}\bigr)
\bigl(\mathcal{N}_i-q^{\widetilde{\alpha}_{i^*}(\xi)/2}-
q^{-\widetilde{\alpha}_{i^*}(\xi)/2}\bigr),
\end{split}
\end{equation*}
see \cite[\S 4]{StQ}. This shows that \eqref{AWde} is correct for case
{\bf a}. 

\noindent
{\bf Case b.} Fix generic $\xi\in E_{\mathbb{C}}$ and write
\[
\mathcal{L}_i:=\bigl(\mathcal{N}_i+q^{\widetilde{\alpha}_{i^*}(\xi)/2}+
q^{-\widetilde{\alpha}_{i^*}(\xi)/2}+\textup{cst}_i\bigr)
\bigl(\mathcal{N}_i-q^{\widetilde{\alpha}_{i^*}(\xi)/2}
-q^{-\widetilde{\alpha}_{i^*}(\xi)}\bigr),
\]
so that $\mathcal{L}_i\Phi_i(\cdot,\xi)=0$.
By \cite[\S 5]{StQ}
there exists a unique meromorphic function $g(x)$ 
in $x\in\mathbb{C}$ satisfying
\begin{equation}\label{d}
g(x)=\frac{\mathcal{W}_i(x,\widetilde{\alpha}_{i^*}(\xi))}
{\mathcal{S}_i(x)\widetilde{\mathcal{S}}(\xi)}\sum_{r=0}^{\infty}g_rq^{-rx}
\qquad (g_0=\Gamma_0(\xi))
\end{equation}
with the series converging normally for $x$ in compacta of $\mathbb{C}$, 
such that 
\[
\mathcal{N}_ig=\bigl(q^{\widetilde{\alpha}_{i^*}(\xi)/2}
+q^{-\widetilde{\alpha}_{i^*}(\xi)/2}\bigr)g.
\]
Then $g$ also satisfies $\mathcal{L}_ig=0$, and, together with
the asymptotic properties \eqref{d} of $g$, this eigenvalue equation 
characterizes $g$. Hence $g=\Phi_i(\cdot,\xi)$. We conclude
that $\Phi_i(\cdot,\xi)$ satisfies \eqref{b1}. 
As for case {\bf a}, this implies that $\Phi_i(\cdot,\xi)$ 
also satisfies the desired eigenvalue equation \eqref{AWde}.
\end{proof}
\begin{rema}
The factorization of the Askey-Wilson difference operator appearing in the
above proof relates to quadratic transformation formulas
for the associated eigenfunctions, see \cite[\S 5]{StQ}. 
In fact, the quadratic 
transformation formula \cite[5.1]{StQ} guarantees the existence of the
function $g$.
\end{rema}

\subsection{Explicit expressions}
For generic $b_j$
the ${}_{r+1}\phi_r$ basic hypergeometric series is defined by
\[
{}_{r+1}\phi_r\left(
\begin{matrix} a_1,a_2,\ldots,a_{r+1}\\ b_1,b_2,\ldots,b_r\end{matrix};
q,z\right):=
\sum_{j=0}^{\infty}\frac{\bigl(a_1,a_2,\ldots,a_{r+1};q\bigr)_j}
{\bigl(q,b_1,\ldots,b_r;q\bigr)_j}z^j,\qquad |z|<1,
\]
where $\bigl(a_1,\ldots,a_s;q\bigr)_j=
\prod_{r=1}^s\prod_{i=0}^{j-1}(1-a_rq^i)$ for $j\in\mathbb{Z}_{\geq 0}\cup
\{\infty\}$ (empty products are equal to one by convention).
The very-well-poised
${}_8\phi_7$ series is defined by
\begin{equation*}
\begin{split}
{}_8W_7\bigl(\alpha_0;\alpha_1,\alpha_2,\alpha_3,\alpha_4,\alpha_5;q,z\bigr)&=
{}_8\phi_7\left(\begin{matrix}
\alpha_0,q\alpha_0^{\frac{1}{2}},-q\alpha_0^{\frac{1}{2}},
\alpha_1,\ldots,\alpha_5\\
\alpha_0^{\frac{1}{2}},-\alpha_0^{-\frac{1}{2}},q\alpha_0/\alpha_1,\ldots,
q\alpha_0/\alpha_5\end{matrix};q,z\right)\\
&=\sum_{r=0}^{\infty}\frac{1-\alpha_0q^{2r}}{1-\alpha_0}z^r
\prod_{j=0}^5\frac{\bigl(\alpha_j;q\bigr)_r}
{\bigl(q\alpha_0/\alpha_j;q\bigr)_r},\qquad |z|<1.
\end{split}
\end{equation*}
If $z=\alpha_0^2q^2/\alpha_1\alpha_2\alpha_3\alpha_4\alpha_5$ 
(which is the case below) then
it has a meromorphic continuation to $(\alpha_0,\ldots,\alpha_5)\in
\bigl(\mathbb{C}^*\bigr)^6$ by \cite[III.36]{GR}, for which we will use
the same notation.

The results from the previous subsection characterize $\Phi_i(\cdot,\xi)$
as asymptotically free eigenfunction of the AW second order difference
operator $\mathcal{M}_i$. In \cite[Prop. 2.2]{StQ} an explicit expression
of this eigenfunction has been obtained (it essentially
traces back to \cite{IR}). In our present notations, the result is as follows.
\begin{prop}\label{explicitPhi}
We have
\begin{equation*}
\begin{split}
&\Phi_{i}(x,\xi)=\mathcal{W}_i(x,\widetilde{\alpha}_{i^*}(\xi))
\frac{\Gamma_0(\xi)}{\mathcal{S}_i(x)\widetilde{\mathcal{S}}(\xi)}\\
&\,\,\times
\frac{\bigl(\frac{q_i^2a_i}{\widetilde{a}_i}q^{-x-\widetilde{\alpha}_{i^*}(\xi)},
\frac{q_i^2b_i}{\widetilde{a}_i}q^{-x-\widetilde{\alpha}_{i^*}(\xi)},
\frac{q_i^2c_i}{\widetilde{a}_i}q^{-x-\widetilde{\alpha}_{i^*}(\xi)},
\frac{q_i^2\widetilde{a}_i}{d_i}q^{-x-\widetilde{\alpha}_{i^*}(\xi)},
d_iq^{-x};q_i^2\bigr)_{\infty}}
{\bigl(
\frac{q_i^4}{d_i}q^{-x-2\widetilde{\alpha}_{i^*}(\xi)};q_i^2\bigr)_{\infty}}\\
&\,\,\times{}_8W_7\Bigl(\frac{q_i^2}{d_i}q^{-x-2\widetilde{\alpha}_{i^*}(\xi)};
\frac{q_i^2}{\widetilde{a}_i}q^{-\widetilde{\alpha}_{i^*}(\xi)},
\frac{q_i^2}{\widetilde{d}_i}q^{-\widetilde{\alpha}_{i^*}(\xi)},
\widetilde{b}_iq^{-\widetilde{\alpha}_{i^*}(\xi)},
\widetilde{c}_iq^{-\widetilde{\alpha}_{i^*}(\xi)},
\frac{q_i^2}{d_i}q^{-x};q_i^2,d_iq^{-x}\Bigr)
\end{split}
\end{equation*}
as meromorphic functions in $(x,\xi)\in\mathbb{C}\times E_{\mathbb{C}}$.
\end{prop}
We are now in a position to determine the connection coefficients
explicitly.

\noindent
{\bf Proof of Theorem \ref{THM4}.} Define meromorphic functions
$n_{\pm}^{s_i}(x,\xi)$ in $(x,\xi)\in\mathbb{C}\times E_{\mathbb{C}}$ by
\begin{equation*}
\begin{split}
n_+^{s_i}(x,\xi)&:=\frac{\mathfrak{e}_{\alpha_i}(x,\widetilde{\alpha}_{i^*}(\xi))-
\widetilde{\mathfrak{e}}_{\alpha_i}(\widetilde{\alpha}_{i^*}(\xi),x)}
{\widetilde{\mathfrak{e}}_{\alpha_i}(\widetilde{\alpha}_{i^*}(\xi),-x)},\\
n_-^{s_i}(x,\xi)&:=\frac{\mathfrak{e}_{\alpha_i}(x,
-\widetilde{\alpha}_{i^*}(\xi))}{\widetilde{\mathfrak{e}}_{\alpha_i}
(\widetilde{\alpha}_{i^*}(\xi),-x)}.
\end{split}
\end{equation*} 
It follows from \cite[Cor. 2.6]{StQ} that
\begin{equation}\label{rankoneversion}
\Phi_i(-x,\xi)=n_+^{s_i}(x,\xi)\Phi_i(x,\xi)+n_-^{s_i}(x,\xi)
\Phi_i(x,s_{i^*}\xi)
\end{equation}
as meromorphic functions in $(x,\xi)\in\mathbb{C}\times E_{\mathbb{C}}$.
Note furthermore that the $n_{\pm}^{s_i}(\cdot,\xi)$ are $2\mu_i$-translation 
invariant. 

For meromorphic functions $f$ and $g$ on $\mathbb{C}$ define
the Wronskian $\lbrack f,g\rbrack$ to be the meromorphic function
\[
\lbrack f,g\rbrack(x):=w_i(x)A_i(x)\bigl(f(x-2\mu_i)g(x)-f(x)g(x-2\mu_i)\bigr),
\]
with $w_i(x)$ the weight function \cite{AW} of the 
Askey-Wilson polynomials,
\[
w_i(x):=\frac{\bigl(q^{2x},q^{-2x};q_i^2\bigr)_{\infty}}
{\bigl(a_iq^{x},a_iq^{-x},b_iq^x,b_iq^{-x},c_iq^x,c_iq^{-x},d_iq^x,d_iq^{-x};
q_i^2\bigr)_{\infty}}.
\]
Since $w_i(x+2\mu_i)A_i(x+2\mu_i)=w_i(x)A_i(-x)$ and 
\begin{equation*}
\begin{split}
\bigl(\mathcal{M}_if\bigr)(x)g(x)-f(x)\bigl(\mathcal{M}_ig\bigr)(x)&=
A_i(x)\bigl(f(x-2\mu_i)g(x)-f(x)g(x-2\mu_i)\bigr)\\
&+A_i(-x)\bigl(f(x+2\mu_i)g(x)-f(x)g(x+2\mu_i)\bigr)
\end{split}
\end{equation*}
it follows that $\lbrack f,g\rbrack(\cdot)$ is $2\mu_i$-translation invariant
if $f$ and $g$ are eigenfunctions of $\mathcal{M}_i$ with the same eigenvalue.
Using the asymptotic expansion of $\Phi_i(\cdot,\xi)$ we thus get
\begin{equation*}
\begin{split}
\lbrack \Phi_i(\cdot,\xi),&\Phi_i(\cdot,s_{i^*}\xi)\rbrack(x)=
\lim_{m\rightarrow\infty}
\lbrack \Phi_i(\cdot,\xi),\Phi_i(\cdot,s_{i^*}\xi)\rbrack(x-2m\mu_i)\\
&=\bigl(q^{-\widetilde{\alpha}_{i^*}(\xi)}-q^{\widetilde{\alpha}_{i^*}(\xi)}\bigr)
\frac{\mathcal{W}_i(x,\widetilde{\alpha}_{i^*}(\xi))
\mathcal{W}_i(x,-\widetilde{\alpha}_{i^*}(\xi))\Gamma_0(\xi)
\Gamma_0(s_{i^*}\xi)\theta(q^{2x};q_i^2)}
{\widetilde{\mathcal{S}}(\xi)\widetilde{\mathcal{S}}(s_{i^*}\xi)
\theta(a_iq^x,b_iq^x,c_iq^x,d_iq^x;q_i^2)}.
\end{split}
\end{equation*}
Write $\Phi_i^-(x,\xi):=\Phi_i(-x,\xi)$.
{}From \eqref{rankoneversion} we now conclude that
\begin{equation*}
\begin{split}
n_{+}^{s_i}(x,\xi)&=
\frac{\lbrack \Phi_i^-(\cdot,\xi),\Phi_i(\cdot,s_{i^*}\xi)\rbrack(x)}
{\lbrack \Phi_i(\cdot,\xi),\Phi_i(\cdot,s_{i^*}\xi)\rbrack(x)},\\
n_-^{s_i}(x,\xi)&=-
\frac{\lbrack \Phi_i^-(\cdot,\xi),\Phi_i(\cdot,\xi)\rbrack(x)}
{\lbrack \Phi_i(\cdot,\xi),\Phi_i(\cdot,s_{i^*}\xi)\rbrack(x)}
\end{split}
\end{equation*}
as meromorphic functions in $(x,\xi)\in\mathbb{C}\times E_{\mathbb{C}}$.
On the other hand, since $m_{\tau_1,\tau_2}^{s_i}(\cdot,\xi)$ is 
$\widetilde{\Lambda}$-translation invariant and 
$\alpha_i(\widetilde{\alpha}_i)=2\mu_i$, we have by Proposition
\ref{XYX}{\bf (i)},
\begin{equation*}
\begin{split}
m_{e,e}^{s_i}(z,\xi)&=
\frac{\lbrack \Phi_i^-(\cdot,\xi),\Phi_i(\cdot,s_{i^*}\xi)\rbrack(\alpha_i(z))}
{\lbrack \Phi_i(\cdot,\xi),\Phi_i(\cdot,s_{i^*}\xi)\rbrack(\alpha_i(z))},\\
m_{s_{i^*},e}^{s_i}(z,\xi)&=-
\frac{\lbrack \Phi_i^-(\cdot,\xi),\Phi_i(\cdot,\xi)\rbrack(\alpha_i(z))}
{\lbrack \Phi_i(\cdot,\xi),\Phi_i(\cdot,s_{i^*}\xi)\rbrack(\alpha_i(z))}
\end{split}
\end{equation*}
as meromorphic functions in $(z,\xi)\in E_{\mathbb{C}}\times E_{\mathbb{C}}$.
Hence
\[
m_{e,e}^{s_i}(z,\xi)=n_+^{s_i}(\alpha_i(z),\xi),\qquad
m_{s_{i^*},e}^{s_i}(z,\xi)=n_-^{s_i}(\alpha_i(z),\xi)
\]
as meromorphic functions in $(z,\xi)\in E_{\mathbb{C}}\times E_{\mathbb{C}}$,
which completes the proof of the theorem.
$\qquad\qquad\qquad\qquad\qquad\qquad\qquad\qquad\qquad
\qquad\qquad\qquad\qquad\qquad\qquad\qquad\quad\Box$

\section{Higher rank addition formula for theta functions}
\label{qcsection}

Recall from Subsection \ref{csection} that
the basic hypergeometric function $\mathcal{E}_{sph}$
is a distinguished Weyl group invariant
solution $\mathcal{E}_{sph}\in\textup{Sol}_{RMKC}^{W_0\times W_0}$
of the bispectral problem of the RMKC operators in case the root system
datum is of the form $D=(R_0,\Delta_0,t,\Lambda,\Lambda)$ and $\kappa_a>0$
for all $a\in R$. It has 
an explicit $c$-function expansion (see \eqref{cfunctionexpansion})
with quantum $c$-function $\mathfrak{c}_{sph}$ as defined in 
Definition \ref{quantumc}. In particular, the quantum $\mathfrak{c}_{sph}$
satisfies \eqref{relationsc}. In this section we investigate
the equations \eqref{relationsc} directly.

\begin{rema}\label{notationsrelations}
The notations in \cite{StSph} are matched to the present ones as follows:
$t,\gamma,k_a,\gamma_0$ in \cite{StSph} correpond to $q^z,q^{-\xi},q^{\kappa_a},
q^{-\rho}$ (where $q^z$ is viewed as element of the complex algebraic
torus $\textup{Hom}(\Lambda,\mathbb{C}^*)$ by $\lambda\mapsto
q^{(\lambda,z)}$), and the basic Harish-Chandra series $\Phi(t,\gamma)$
in \cite[4.6]{StSph} matches with our renormalized Harish-Chandra series
\[
\Bigl(\sum_{w\in W_0}q^{2\kappa_w}\Bigr)^{-1}\mathcal{W}(z,\xi)^{-1}
\frac{\vartheta_\Lambda(z+w_0\xi)}{\vartheta_\Lambda(\rho-z)\vartheta(\rho-\xi)}
\Phi(z,\xi).
\]
\end{rema}

Let $(D,\kappa,q)$ be an arbitrary choice of initial datum.
Let $\Xi\in\mathcal{M}$ and write
\[
\mathfrak{c}_\Xi(z,\xi):=
\frac{\Xi(z,\xi)}{\mathcal{W}(z,\xi)}
\prod_{\alpha\in R_0^+}\frac{\theta\bigl(\widetilde{a}_\alpha 
q^{\widetilde{\alpha}(\xi)},\widetilde{b}_\alpha q^{\widetilde{\alpha}(\xi)},
\widetilde{c}_\alpha q^{\widetilde{\alpha}(\xi)},
\widetilde{d}_\alpha q^{\widetilde{\alpha}(\xi)};q_\alpha^2\bigr)}
{\theta\bigl(q^{2\widetilde{\alpha}(\xi)};q_\alpha^2\bigr)}.
\]
Using $q^{2(\lambda,\widetilde{\rho})}=\prod_{\alpha\in R_0^+}
a_\alpha^{(\lambda,\alpha^\vee)}$ and the functional equation \eqref{fe},
we have
$\mathfrak{c}_\Xi\in\mathcal{F}$ if and only if
\begin{equation}\label{quasiinvariance}
\begin{split}
\Xi(z+\mu,\xi)&=q^{(\rho-\xi,w_0\mu)}\Xi(z,\xi)\qquad\forall\,
\mu\in\widetilde{\Lambda},\\
\Xi(z,\xi+\lambda)&=q^{(\lambda,\widetilde{\rho}-w_0z)}\Xi(z,\xi)\qquad
\forall\,\lambda\in\Lambda
\end{split}
\end{equation}
(use for the theta function factors of $c_\Xi$ 
corresponding to $\alpha\in R_0^+$
with
$\bigl(\Lambda,\alpha^\vee\bigr)=\mathbb{Z}$ the fact that
$\widetilde{c}_\alpha=q_\alpha\widetilde{a}_\alpha$ and 
$\widetilde{d}_\alpha=q_\alpha\widetilde{b}_\alpha$). 
\begin{rema}
In the twisted equal lattice case $D=(R_0,\Delta,t,\Lambda,\Lambda)$
we have $\psi=\theta$ and $\mathfrak{c}_{sph}=\mathfrak{c}_{\Xi_{sph}}$ with 
\begin{equation}\label{xisph}
\Xi_{sph}(z,\xi):=\frac{\vartheta_\Lambda(\rho+(\kappa_{2a_0}-\kappa_0)
\delta_s^\vee+z+w_0\xi)}
{\vartheta_\Lambda((\kappa_{2a_0}-\kappa_0)\delta_s^\vee+z)
\vartheta_\Lambda((\kappa_{2a_0}-\kappa_{2\psi})\delta_s^\vee-\xi)}.
\end{equation}
The quasi-invariance properties \eqref{quasiinvariance}
for $\Xi_{sph}$ can be directly checked using the functional equations
\[
\vartheta_\Lambda(z+\lambda)=
q^{-\frac{|\lambda|^2}{2}}q^{-(\lambda,z)}\vartheta_\Lambda(z)
\qquad
\forall\, \lambda\in\Lambda
\]
and noting that 
$\rho+(\kappa_{2\psi}-\kappa_0)\delta_s^\vee=\widetilde{\rho}$.
\end{rema}
\begin{lem}\label{BBB}
Let $i\in\{1,\ldots,n\}$ such that $\bigl(\Lambda,\alpha_i^\vee\bigr)=
\mathbb{Z}=\bigl(\widetilde{\Lambda},\widetilde{\alpha}_i^\vee\bigr)$.
Suppose $\Xi\in\mathcal{M}$
satisfies \eqref{quasiinvariance}. Then $\mathfrak{c}_{\Xi}\in\mathcal{F}$
satisfies \eqref{relationsc}, i.e.
\begin{equation}\label{relationscXi}
\mathfrak{c}_{\Xi}(z,\xi)=m_{e,e}^{s_i}(z,\xi)\mathfrak{c}_{\Xi}(s_iz,\xi)+
m_{s_{i^*},e}^{s_i}(z,s_{i^*}\xi)\mathfrak{c}_{\Xi}(s_iz,s_{i^*}\xi),
\end{equation}
if and only if
\begin{equation}\label{towardsRiemann}
\begin{split}
\theta\bigl(q^{\widetilde{\alpha}_{i^*}(\xi)},q^{2\kappa_i-\alpha_i(z)};q_i\bigr)
\Xi(z,\xi)
&=\theta\bigl(q^{2\kappa_i},q^{\widetilde{\alpha}_{i^*}(\xi)-\alpha_i(z)};q_i\bigr)
\Xi(s_iz,\xi)\\
&-q^{\widetilde{\alpha}_{i^*}(\xi)}
\theta\bigl(q^{2\kappa_i-\widetilde{\alpha}_{i^*}(\xi)},
q^{-\alpha_i(z)};q_i\bigr)\Xi(s_iz,s_{i^*}\xi).
\end{split}
\end{equation}
\end{lem}
\begin{proof}
This follows from a straightforward computation using the 
expressions of the connection coefficients $m_{e,e}^{s_i}$
and $m_{s_{i^*},e}^{s_i}$ from Proposition \ref{cUScase} and using
that
\[
(\widetilde{a}_i,\widetilde{b}_i,\widetilde{c}_i,\widetilde{d}_i)=
(q^{2\kappa_i},-1,q_iq^{2\kappa_i},-q_i).
\]
\end{proof}
Lemma \ref{BBB} 
leads to the root system analog (Proposition \ref{addition}) of the 
addition formula \eqref{Rid} for Jacobi theta functions.

\noindent
{\bf Proofs of Proposition \ref{addition}.}\\
{\bf 1.}
It suffices to show that the equation \eqref{relationsc}
for $\Xi_{sph}$ implies that
\eqref{Ridroot} is valid with $\rho$ replaced by 
$\rho+(\kappa_{2a_0}-\kappa_0)\delta_s^\vee$. 
This follows immediately from the explicit expression \eqref{xisph}
of $\Xi_{sph}$ and the $W_0$-invariance of $\vartheta_\Lambda$ if
$\bigl(\Lambda,\alpha^\vee\bigr)=\mathbb{Z}$ for all $\alpha\in R_0$.
If the condition $\bigl(\Lambda,\alpha^\vee\bigr)=\mathbb{Z}$ 
is not valid for all $\alpha\in R_0$, 
then $\bigl(\Lambda,\alpha^\vee\bigr)=2\mathbb{Z}$
for short roots $\alpha\in R_0$ (see \cite{StSph}).
Hence $\alpha_i$ has to be a long root,
consequently $s_i(\delta_s^\vee)=\delta_s^\vee$. The result now follows
again by the explicit expression \eqref{xisph} of $\Xi_{sph}$ and the
$W_0$-invariance of $\vartheta_\Lambda$.\\
{\bf 2.} This second proof is by direct analytical methods.
Fix generic $\xi\in E_{\mathbb{C}}$ and write $g_\xi(z)$
for the right hand side of \eqref{Ridroot}. It is a holomorphic function
in $z\in E_{\mathbb{C}}$ which vanishes if
\[
\alpha_i(z)=2\kappa_i+\mu_ik+\frac{2\pi\sqrt{-1}}{\tau}m\qquad
(k,m\in\mathbb{Z}).
\]
Hence
\[
f_\xi(z):=\frac{g_\xi(z)}{\theta\bigl(q^{2\kappa_i-\alpha_i(z)};q_i\bigr)}
\]
is holomorphic in $z\in E_{\mathbb{C}}$. 
Let $\Lambda^\vee$ be the dual lattice of $\Lambda$ in $E$ with respect to
the scalar product $\bigl(\cdot,\cdot\bigr)$. By a direct computation
it follows that $f_\xi$ satisfies
\begin{equation*}
\begin{split}
f_\xi(z+\nu)&=f_\xi(z),
\qquad\qquad\qquad\qquad\qquad\,\,
\forall\,\nu\in\frac{2\pi\sqrt{-1}}{\tau}\Lambda^\vee,\\
f_\xi(z+\lambda)&=q^{-|\lambda|^2}q^{-(\lambda,\rho+z+w_0\xi)}f_\xi(z),
\qquad \forall\, \lambda\in\Lambda.
\end{split}
\end{equation*}
This implies that $f_\xi(z)=\textup{cst}_\xi\vartheta_\Lambda\bigl(
\rho+z+w_0\xi\bigr)$ for some constant $\textup{cst}_\xi$. Setting
$z=0$ shows that $\textup{cst}_\xi=
\theta\bigl(q^{\widetilde{\alpha}_{i^*}(\xi)}\bigr)$.
$\qquad\qquad\qquad\qquad\qquad\qquad\qquad\qquad\qquad
\qquad\qquad\,\,\,\Box$
\begin{rema}\label{relationRid}
To see how the addition formula \eqref{Rid} for the Jacobi theta
function can be recovered from
\eqref{Ridroot}, take $D=(R_0,\Delta,t,Q,Q)$ 
with the root system $R_0$ of type $B_n$ ($n\geq 2$) realized 
within the Euclidean space $E=\mathbb{R}^n$ with orthonormal basis
$\{\epsilon_i\}_{i=1}^n$ as 
$\{\pm \epsilon_i\pm\epsilon_j\}_{1\leq i<j\leq n}\cup\{\pm\epsilon_i\}_{i=1}^n$.
We take $\Delta_0=(\epsilon_1-\epsilon_2,\ldots,
\epsilon_{n-1}-\epsilon_n,\epsilon_n)$ as ordered basis of $R_0$.
Since $Q=\bigoplus_{i=1}^n\mathbb{Z}\epsilon_i$, the Jacobi triple product
identity implies that
\[
\vartheta_Q(z)=\bigl(q;q\bigr)_{\infty}^n\prod_{j=1}^n
\theta\bigl(-q^{\frac{1}{2}+z_j};q\bigr),
\]
where $z_j:=\epsilon_j(z)$. Then \eqref{Ridroot} for $1\leq i<n$
is easily seen to reduce to \eqref{Rid}.
\end{rema}

The following lemma gives yet another reformulation of \eqref{relationsc}.
\begin{lem}
Suppose that $\Xi\in\mathcal{M}$ satisfying
\eqref{quasiinvariance}. Then $c_\Xi\in\mathcal{F}$ satisfies
\eqref{relationscXi} if and only if
\begin{equation*}
\begin{split}
&\frac{\theta\bigl(\widetilde{d}_iq^{\widetilde{\alpha}_{i^*}(\xi)},
d_iq^{-\alpha_i(z)};q_i^2\bigr)}
{\theta\bigl(\frac{\widetilde{d}_i}{a_i}q^{\widetilde{\alpha}_{i^*}(\xi)-
\alpha_i(z)};q_i^2\bigr)}\Xi(z,\xi)-
\frac{\theta\bigl(\widetilde{d}_iq^{\widetilde{\alpha}_{i^*}(\xi)},
d_iq^{\alpha_i(z)};q_i^2\bigr)}
{\theta\bigl(\frac{\widetilde{d}_i}{a_i}q^{\widetilde{\alpha}_{i^*}(\xi)+
\alpha_i(z)};q_i^2\bigr)}\Xi(s_iz,\xi)=\\
&=\frac{\mathfrak{e}_{\alpha_i}(-\alpha_i(z),-\widetilde{\alpha}_{i^*}(\xi))}
{\widetilde{\mathfrak{e}}_{\alpha_i}
(-\widetilde{\alpha}_{i^*}(\xi),-\alpha_i(z))}\\
&\times\left(
\frac{\theta\bigl(\widetilde{d}_iq^{-\widetilde{\alpha}_{i^*}(\xi)},
d_iq^{\alpha_i(z)};q_i^2\bigr)}
{\theta\bigl(\frac{\widetilde{d}_i}{a_i}q^{-\widetilde{\alpha}_{i^*}(\xi)+
\alpha_i(z)};q_i^2\bigr)}\Xi(s_iz,s_{i^*}\xi)-
\frac{\theta\bigl(\widetilde{d}_iq^{\widetilde{\alpha}_{i^*}(\xi)},
d_iq^{\alpha_i(z)};q_i^2\bigr)}
{\theta\bigl(\frac{\widetilde{d}_i}{a_i}q^{\widetilde{\alpha}_{i^*}(\xi)+
\alpha_i(z)};q_i^2\bigr)}\Xi(s_iz,\xi)\right).
\end{split}
\end{equation*}
\end{lem}
\begin{proof}
By \cite[Cor. 2.8]{StQ},
\[
m_{e,e}^{s_i}(z,\xi)=
\frac{\mathfrak{e}_{\alpha_i}(\alpha_i(z),\widetilde{\alpha}_{i^*}(\xi))
-\mathfrak{e}_{\alpha_i}(-\alpha_i(z),-\widetilde{\alpha}_{i^*}(\xi))
m_{s_{i^*},e}^{s_i}(z,s_{i^*}\xi)}{\mathfrak{e}_{\alpha_i}(-\alpha_i(z),
\widetilde{\alpha}_{i^*}(\xi))}
\]
(this identity boils down to the $W_0$-invariance 
of the Askey-Wilson function,
which is the basic hypergeometric function associated to 
$D=(R_0,\Delta_0,t,Q,Q)$ with $R_0$ of type $A_1$). Combined with
the explicit expression of the connection coefficient
$m_{s_{i^*},e}^{s_i}$ (see Theorem \ref{THM4}), we conclude that 
\eqref{relationscXi} is equivalent to
\begin{equation*}
\begin{split}
&c_\Xi(z,\xi)-\frac{\mathfrak{e}_{\alpha_i}(\alpha_i(z),
\widetilde{\alpha}_{i^*}(\xi))}
{\mathfrak{e}_{\alpha_i}(-\alpha_i(z),\widetilde{\alpha}_{i^*}(\xi))}
c_\Xi(s_iz,\xi)=\\
&=\frac{\mathfrak{e}_{\alpha_i}(\alpha_i(z),\widetilde{\alpha}_{i^*}(\xi))}
{\widetilde{\mathfrak{e}}_{\alpha_i}(-\widetilde{\alpha}_{i^*}(\xi),
-\alpha_i(z))}\left(
c_\Xi(s_iz,s_{i^*}\xi)-\frac{\mathfrak{e}_{\alpha_i}(-\alpha_i(z),
-\widetilde{\alpha}_{i^*}(\xi))}{\mathfrak{e}_{\alpha_i}(-\alpha_i(z),
\widetilde{\alpha}_{i^*}(\xi))}c_\Xi(s_iz,\xi)\right).
\end{split}
\end{equation*}
Straightforward simplifications now complete the proof of the lemma.
\end{proof}
\begin{cor}
If $\Xi\in\mathcal{M}$ satisfies \eqref{quasiinvariance} and 
\begin{equation}\label{alternative}
\begin{split}
\Xi(s_iz,\xi)&=\frac{\theta\bigl(\frac{\widetilde{d}_i}{a_i}
q^{\widetilde{\alpha}_{i^*}(\xi)+\alpha_i(z)},d_iq^{-\alpha_i(z)};q_i^2\bigr)}
{\theta\bigl(\frac{\widetilde{d}_i}{a_i}q^{\widetilde{\alpha}_{i^*}(\xi)-
\alpha_i(z)},d_iq^{\alpha_i(z)};q_i^2\bigr)}\Xi(z,\xi),\\
\Xi(z,s_{i^*}\xi)&=\frac{\theta\bigl(\frac{\widetilde{d}_i}{a_i}
q^{-\widetilde{\alpha}_{i^*}(\xi)-\alpha_i(z)},\widetilde{d}_i
q^{\widetilde{\alpha}_{i^*}(\xi)};q_i^2\bigr)}
{\theta\bigl(\frac{\widetilde{d}_i}{a_i}q^{\widetilde{\alpha}_{i^*}(\xi)-
\alpha_i(z)},\widetilde{d}_iq^{-\widetilde{\alpha}_{i^*}(\xi)};q_i^2\bigr)}
\Xi(z,\xi),
\end{split}
\end{equation}
then $c_\Xi\in\mathcal{F}$ satisfies \eqref{relationscXi}.
\end{cor}
For $D=(R_0,\Delta_0,t,\Lambda,\Lambda)$ and $i\in\{1,\ldots,n\}$
such that $\bigl(\Lambda,\alpha_i^\vee\bigr)=\mathbb{Z}$, 
the solution $\Xi_{sph}$ (see \eqref{xisph}) of \eqref{relationscXi} 
does not satisfy \eqref{alternative}; in that case, one has to resort to
the root system analog \eqref{Ridroot} of the addition formula
for the Jacobi theta function for a
direct proof of \eqref{relationscXi}. On the other hand,
for the Koornwinder root system datum (see Remark \ref{relationRid}), it can
be directly checked that $\Xi_{sph}$ satisfies \eqref{alternative}
for $i=n$ (here we use the conventions from Remark
\ref{relationRid}, i.e. $\alpha_n\in\Delta_0$ is the short simple
root), thus providing a direct proof 
that $\Xi_{sph}$ satisfies \eqref{relationscXi} for $i=n$.


\end{document}